\documentclass[12pt]{amsart}

\pdfoutput=1

\usepackage[hmarginratio={1:1}]{geometry}
\usepackage[shortlabels]{enumitem}
\usepackage[hidelinks]{hyperref}
\usepackage{tikz-cd}
\usepackage{amssymb}
\usepackage{mathtools}
\usepackage{float}
\usepackage{euscript}
\usepackage{soul}

\usetikzlibrary{decorations.pathreplacing}

\newcommand\abs[1]{\lvert{}#1\rvert{}}
\newcommand\Ac{\mathcal{A}}
\newcommand\bott{\mathrm{bot}}
\newcommand\br[1]{\langle{}#1\rangle{}}
\newcommand\cc{\mathcal{C}}
\newcommand\CC{\mathrm{CC}}
\newcommand\CH{\mathrm{CH}}
\newcommand\ch{\mathrm{c}}
\newcommand\contr{\mathrm{contr}}
\newcommand\C{\mathbb{C}}
\newcommand\defeq{\mathrel{\vcenter{\baselineskip0.5ex \lineskiplimit0pt \hbox{.}\hbox{.}}} =}
\newcommand\defword[1]{\ul{#1}}
\newcommand\dr{\dv{r}}
\newcommand\ds{\dv{s}}
\newcommand\dt{\dv{t}}
\newcommand\du{\dv{u}}
\newcommand\dv{\mathrm{d}}
\newcommand\dx{\dv{x}}
\newcommand\dy{\dv{y}}
\newcommand\dz{\dv{z}}
\newcommand\enumbelow{\ }
\newcommand\floor[1]{\lfloor{#1}\rfloor}
\newcommand\Hm{\operatorname{H}}
\newcommand\hto{\hookrightarrow}
\newcommand\id{\mathrm{id}}
\newcommand\im{\operatorname{im}}
\newcommand\inc{\mathrm{inc}}
\newcommand\intprod{\mathbin{\lrcorner}}
\newcommand\I{^{-1}}
\newcommand\killenumspace{\ \\\vspace{-10pt}}
\newcommand\K{\mathcal{K}}
\newcommand\Ldv{\mathcal{L}}
\newcommand\loc{\mathrm{loc}}
\newcommand\Mbar{\overline{\M}}
\newcommand\M{\EuScript{M}}
\newcommand\new{\mathrm{new}}
\newcommand\old{\mathrm{old}}
\newcommand\Op{\mathcal{O}p}
\renewcommand\phi{\varphi}
\newcommand\plr[1]{\left(#1\right)}
\newcommand\pt{\mathrm{pt}}
\newcommand\Q{\mathbb{Q}}
\newcommand\Rb{\mathrm{R}}
\newcommand\rIII{\mathrm{III}}
\newcommand\rII{\mathrm{II}}
\newcommand\rIV{\mathrm{IV}}
\newcommand\rI{\mathrm{I}}
\newcommand\rstr[2]{{\left.#1\right|_{#2}}}
\newcommand\R{\mathbb{R}}
\newcommand\sch[1]{\{#1\}}
\newcommand\set[1]{\{#1\}}
\newcommand\SSS{\EuScript{S}}
\newcommand\std{\mathrm{std}}
\newcommand\sut{\mathrm{s}}
\newcommand\s{\mathrm{S}}
\newcommand\Tn{\mathrm{T}}
\newcommand\toi{\xrightarrow{_\sim}}
\newcommand\topp{\mathrm{top}}
\newcommand\tox[2][]{\xrightarrow[#1]{#2}}
\newcommand\toxlim[1]{\tox{#1\to\infty}}
\newcommand\T{\mathbb{T}}
\newcommand\vir{\mathrm{vir}}
\newcommand\wh[1]{\widehat{#1}}
\newcommand\wt[1]{\widetilde{#1}}
\newcommand\Z{\mathbb{Z}}

\theoremstyle{definition}
\newtheorem{defn}{Definition}[section]
\newtheorem{rmk}[defn]{Remark}

\theoremstyle{remark}

\theoremstyle{plain}
\newtheorem{thm}[defn]{Theorem}
\newtheorem{propn}[defn]{Proposition}
\newtheorem{lem}[defn]{Lemma}
\newtheorem{cor}[defn]{Corollary}

\newtheoremstyle{dotless}{}{}{}{}{\bfseries}{}{ }{\thmname{#1}\thmnumber{ #2}\thmnote{#3}}
\theoremstyle{dotless}
\newtheorem{setup}{Setup}

\let\oldproofname=\proofname
\renewcommand{\proofname}{\rm\bf{\oldproofname}}

\usepackage[style=alphabetic,
            doi=false,
            isbn=false,
            maxnames=99,
            maxalphanames=99,
            backref=true]{biblatex}
\DefineBibliographyStrings{english}{%
backrefpage = {p\adddot},%
backrefpages = {pp\adddot}%
}
\DeclareFieldFormat{pagerefformat}{{\tiny[#1]}}
\renewbibmacro*{pageref}{%
\iflistundef{pageref}
{}
{\printtext[pagerefformat]{%
    \ifnumgreater{\value{pageref}}{1}
    {\bibstring{backrefpages}}
    {\bibstring{backrefpage}}%
    \printlist[pageref][-\value{listtotal}]{pageref}}}}

\title{Exotic tight contact structures on $\R^n$}
\author{François-Simon Fauteux-Chapleau and Joseph Helfer}
\date{}

\begin{document}
\maketitle

\begin{abstract}
We introduce a variant of contact homology for convex open contact manifolds. As an application, we prove the existence of (in fact, infinitely many) exotic tight contact structures on $\R^{2n-1}$ for all $n>2$.
\end{abstract}

\tableofcontents

\section{Introduction}
The main goal of this paper is to extend contact homology (see \cite{eliashberg-invariants_in_contact,eliashberg-givental-hofer-intro-to-sft,pardon-contact-homology-and-vfc}), which is normally defined for \emph{closed} (co-oriented) contact manifolds, to certain open contact manifolds, namely those that are \emph{convex} in the sense of \cite[\S3.5]{eliashberg-gromov-convex}. As an application, we prove:

{
\begin{thm}[see Theorem~\ref{thm:exotic}]\label{thm:exotic-intro}
  For each $n>2$, there exists a tight contact structure on $\R^{2n-1}$ that is not isomorphic to $\R^{2n-1}_\std$. In fact, there exist infinitely many pairwise non-isomorphic tight contact structures on $\R^{2n-1}$.
\end{thm}
}
We also obtain a similar result for certain hypertight contact manifolds (see Theorem~\ref{thm:hypertight}).

We note that the statement of Theorem~\ref{thm:exotic-intro} is false for $n=2$ by \cite[Theorem~2.1.4]{eliashberg-20-years}.
A proof of part of Theorem~\ref{thm:exotic-intro} (one exotic contact structure in $\R^n$ for $n\equiv 1\ \mathrm{mod}\ 4$) has since been given independently by Uljarević \cite{uljarevic-squeezing} using different methods: there, the theorem is deduced from a non-squeezing property for balls in Brieskorn spheres (see below), which in turn is proven using a new version of symplectic homology.

Before discussing how this theorem is proved, let us give some more general context. A vector field $X$ in a contact manifold $(V,\xi)$ is \emph{contact} if the flow (locally) defined by $X$ is by contactomorphisms. A hypersurface $\Sigma\subset V$ in a contact manifold $(V,\xi)$ is \emph{convex} if there exists a contact vector field $X$ in a neighbourhood of $\Sigma$ which is transverse to $\Sigma$. A contact manifold has \emph{convex boundary} if its boundary is a convex hypersurface, and a contact manifold is \emph{convex} if it is isomorphic to the interior of a compact contact manifold with convex boundary.%
\footnote{We use a different but equivalent definition of convexity below; see Remark~\ref{rmk:convexity-definitions}.}

We emphasize that while the version of contact homology in this paper is only defined for open contact manifolds which \emph{can be} presented as the interior of a contact manifold with convex boundary, it is an invariant of the open contact manifold alone, and does not depend on such a presentation.
This crucial fact is what allows us to use this invariant to distinguish contact structures on open manifolds such as $\R^n$, and the bulk of the technical work in this paper is dedicated to proving it.

In \cite{colin-others-sutures-and-ch}, contact homology is defined for so-called \emph{sutured contact manifolds}, which are certain contact manifolds with boundary and corners.
Sutured contact homology, as defined, depends on the convex structure of the boundary, and hence cannot be used to distinguish open contact manifolds.
On the other hand, it seems plausible that sutured contact homology is isomorphic to the invariant defined in the present paper (see the discussion in \S\ref{subsec:sutured}).

We note that \cite{eliashberg-kim-polterovich-squeezing} also define contact homology for a certain class of open contact manifolds, but those ones are \emph{not} convex.

Returning to Theorem~\ref{thm:exotic-intro}, the exotic contact structures in question are obtained by removing a point from certain ``Brieskorn spheres'', which are themselves known to have exotic contact structures (see \cites{ustilovsky-infinitely-many-spheres}[\S1.9.3]{eliashberg-givental-hofer-intro-to-sft}). As is done in \cite{ustilovsky-infinitely-many-spheres}, we show that the resulting contact structure is non-standard by showing it has non-trivial contact homology (unlike the standard contact structure). This is done via a computation analogous to the one showing that the Brieskorn sphere has non-standard contact homology, where we must now also analyze (to the extent needed to show non-triviality) how the contact homology changes upon removing a point (see \S\ref{sec:exotic-contact-strucs}).

Now let us discuss the adjective ``tight'' used in Theorem~\ref{thm:exotic-intro}. A basic dichotomy for 3-dimensional contact manifolds is between ``tight'' and ``overtwisted'' ones (see \cite[\S2]{eliashberg-invariants_in_contact}). For open manifolds, the overtwisted structures are further divided into those that are overtwisted ``at infinity'' (i.e., when restricted to every complement of a compact set) or ``tight at infinity''. The basic result about overtwisted (or, in the open case, overtwisted at infinity) contact structures is that they are ``flexible'' (or satisfy an ``h-principle'') in the sense that each homotopy class of hyperplane fields contains a unique overtwisted (or overtwisted-at-infinity) representative, whereas the tight ones are ``rigid'' -- a homotopy class of hyperplane fields can contain 0 or multiple tight contact structures.

The notion of overtwistedness was extended to higher-dimensional contact manifolds in a natural way in \cite{borman-eliashberg-murphy-overtwisted}, and in particular they prove an analogous homotopy-classification result (so that, in particular, there exist overtwisted contact structures on $\R^{2n-1}$), and also that the standard contact structure on $\R^{2n-1}$ is tight (the corresponding fact for $\R^3$ is \emph{Bennequin's theorem}, \cite[\S2]{eliashberg-invariants_in_contact}).

Since overtwisted structures are known to exist, the novelty of Theorem~\ref{thm:exotic-intro} is that exotic \emph{tight} contact structures exist -- or rather, that the known tight contact structures, obtained by puncturing Brieskorn manifolds, are exotic.

\textbf{Acknowledgments:}
We thank Yasha Eliashberg for suggesting this problem and for his guidance throughout the project. We are also grateful to John Pardon, Kyler Siegel, Chris Wendl, Dylan Cant, François Greer, and Abigail Ward for helpful discussions, and to John Morgan for suggesting a simplification in the computation in \S\ref{sec:exotic-contact-strucs}.
A preliminary version of this paper appeared as part of the second author's Ph.D. thesis.
We also thank the anonymous referees for helpful recommendations which improved the paper.

\subsection{Recollections on contact homology}\label{sec:recollections-on-ch}
We take \cite{pardon-contact-homology-and-vfc} as our reference for the definition of contact homology (of a closed, co-oriented contact manifold). Let us recall the basic elements here.

Given a closed, co-oriented contact manifold $(V,\xi)$, its \emph{contact homology algebra} is a supercommutative $\Q$-superalgebra $\CH_\bullet(V,\xi)$. This object is really only well-defined up to isomorphism, in the sense that the construction involves the choice of some additional data $X$, but associated to any two such choices $X$ and $X'$ is an isomorphism between the resulting objects (and moreover, this gives a ``connected simple system'' in the sense of \cite[\S5.3]{conley-isolated_sets}, i.e., a functor from the category with objects the various $X$ and with a unique (iso)morphism between any two objects).\footnote{As is done in \cite{pardon-contact-homology-and-vfc}, one can then formally obtain a canonical object $\CH_\bullet(C,\xi)$ as a quotient of a disjoint union over all possible $X$, i.e., by taking the colimit of the just-mentioned functor.}

Specifically, the choices involved are: a representing one-form $\alpha$; a complex structure $J$ on $\xi$; and the choice of a so-called ``perturbation datum'' $\theta$, which is needed to define the ``virtual fundamental cycles'' for the moduli spaces coming in to the definition of $\CH_\bullet$ (the details of which will need not concern us). For a fixed choice of such data, the resulting object is denoted $\CH_\bullet(V,\xi)_{\alpha,J,\theta}$.

The proof that these are isomorphic for different choices of $(\alpha,J,\theta)$ results from a more general functoriality of these objects with respect to certain symplectic cobordisms; by taking ``trivial'' cobordisms (but with different choices of $(\alpha,J,\theta)$ on the two ends), one obtains the required isomorphisms.

The only way in which the assumption that $V$ is compact comes into the above construction is when proving that the relevant moduli spaces of holomorphic curves are compact. In order to achieve this in our situation, we will have to restrict our attention to certain ``admissible'' $\alpha$ and $J$, and to a certain limited class of cobordisms, which will in particular only allow us to compare contact forms adapted to the same convex structure. In \S\ref{sec:defining-ch}, we show how to circumvent this and compare different convex structures, by considering certain contact forms which interpolate between two given convex structures.

But first, in \S\ref{sec:preparations}, we will introduce the notions of admissible contact forms and almost complex structures and the appropriate notion of symplectic cobordism, and prove the basic properties we need about them, in particular the compactness results.

We note that the contact homology of a convex open $V$ we define in this paper will be (isomorphic to) $\CH_\bullet(V,\xi)_{\alpha,J,\theta}$ for \emph{any} admissible choice of $\alpha$ (and $J$ and $\theta$).
This is in contrast to the situation with symplectic homology, in which one first introduces the notion of admissible Hamiltonian, but must then pass to a limit over all admissible Hamiltonians.

\subsection{Preliminaries and notation}
All differential-geometric objects (manifolds, differential forms, etc.) are assumed to be smooth unless declared otherwise.

\subsubsection{Contact manifolds}
For basic background about contact structures, see, e.g., \cite{geiges-intro-contact}. By a \emph{contact structure}, we will always mean a \emph{co-oriented} contact structure. Given a contact structure $\xi$, we say that a one-form $\alpha$ \emph{represents} $\xi$ if $\ker(\alpha)=\xi$ and $\alpha$ induces the given co-orientation. Note that for any two representing contact forms $\alpha$ and $\alpha'$, we have that $\alpha=g\alpha'$ for some positive function $g$, which we denote by $\frac{\alpha}{\alpha'}$, and we say that $\alpha\ge\alpha'$ if $\frac{\alpha}{\alpha'}\ge1$.

\subsubsection{Pulled back forms}
Given a manifold $U$ which is a product $\br{\pi,t}\colon U\toi\Sigma\times I$ of a manifold $\Sigma$ with an interval, and a differential form $\alpha$ on $\Sigma$, we may just write $\alpha$ for the pullback $\pi^*\alpha$.

\subsubsection{Convex hypersurfaces and convex structures}
A hypersurface $\Sigma$ in a contact manifold $V$ is \defword{convex} if it is transverse to some contact vector field $Y$.

The subset $\Gamma=\Gamma_{Y}=\set{p\in\Sigma\mid Y_p\in\xi_p}\subset\Sigma$ is called the \defword{dividing set of $\Sigma$}. It is known to be a smooth hypersurface (cut out transversely by $u=Y\intprod\alpha$ for any $\alpha$ representing $\xi$, see e.g. \cite[\S2.2]{colin-others-sutures-and-ch}).

We denote by $\Sigma^{+,Y}$ and $\Sigma^{-,Y}$ (or just $\Sigma^+$ and $\Sigma^-$) the subsets of $\Sigma\setminus\Gamma$ where $Y$ lies, respectively, on the positive or negative side of $\xi$.

A \defword{finite-type convex structure} $(X,t)$ on $V$ is a complete contact vector field $X$ on $V$ together with a smooth proper function $t$ such that $X\intprod\dt=1$ outside of some compact subset. We will omit ``finite type'', as this is the only kind of convex structure we will consider.
If such an $(X,t)$ exists, we say that $V$ is a \defword{convex open contact manifold}.

For sufficiently large $t_0$, each level-set $\Sigma_{t_0}=\set{t=t_0}$ is then a convex hypersurface, and the various $\Sigma_{t_0}$ are canonically isomorphic (via flowing along $X$) for different values of $t_0$. Hence we can identify them all with a fixed manifold $\Sigma$ endowed with a partition $\Sigma^+\cup\Sigma^-\cup\Gamma$.

We will use $\Op_V(\infty)$ (or just $\Op(\infty)$) to denote a subset $\set{t>T}$ (which has compact complement) for sufficiently large $T$ (see \S\ref{subsubsec:symp-cobords} for a related notation). We have a projection $\pi=\pi^\Sigma\colon\Op_V(\infty)\to\Sigma$, as just described, giving a product decomposition $\br{\pi,t}\colon\Op_V(\infty)\to\Sigma\times(T,\infty)$.

We will in general use the notation $\Sigma_{t_0},\Sigma,\Gamma,\pi$ when dealing with a fixed convex structure $(X,t)$ without further explanation.

\begin{rmk}\label{rmk:convexity-definitions}
  As mentioned in the introduction, the above notion of convexity is equivalent to $V$ being the interior of a compact contact manifold $\overline V$ with convex boundary $\Sigma=\partial V$.
  Let us briefly explain how this is established.

  Given $\overline V$, let $\br{\pi,t}\colon\Sigma\times(-\varepsilon,0]\toi\Op(\partial V)$ be a collar neighbourhood of the boundary induced by a transverse contact vector field $X$, so that $X=\partial_t$.
  Recall now that to any ``contact Hamiltonian'' (i.e., smooth function) $H$ on a contact manifold with contact form $\alpha$, there is associated a vector field $X_H$, determined by the conditions $X_H\intprod\alpha=H$ and $\rstr{(X_H\intprod\dv\alpha)}{\ker\alpha}=-\rstr{\dv H}{\ker\alpha}$; and conversely, every contact vector field arises in this way.

  By Lemma~\ref{lem:t-invt-form} below, we have a representing contact form $\alpha=\beta+u\dt$ on $\Op(\partial V)$ with $\beta\in\Omega^{1}(\Sigma)$ and $u\in\Omega^0(\Sigma)$.
  With respect to $\alpha$, the contact Hamiltonian inducing $\partial_t$ is $\partial_t\intprod\alpha=u$.

  But now for any smooth function $\theta(t)$, the contact Hamiltonian $u\cdot\theta$ induces a vector field of the form $\theta\partial_t+Y$ with $Y\intprod\dt=0$ (one can show this by arguing separately (i) away from $\Gamma=u\I(0)$, using the contact form $\alpha'=u\I\alpha=\dt+u\I\beta$ (and using that $\dv(u\I\beta)\in\Omega^2(\Sigma)$ is symplectic), and (ii) near $\Gamma$, using a form $\beta_0+s\dt$ as in Lemma~\ref{lem:normal-form-near-gamma} below).

  This proves the claim, since the vector field $Y$ associated to a non-vanishing function $\theta(t)$ with $\theta(t)=-t$ for $t\in(-\varepsilon/2,0)$ and $\theta(t)=1$ for $t\le-\varepsilon$, together with an appropriate function $s$, gives a convex structure $(Y,s)$ on $V=\operatorname{int}\overline V$.
  (To go in the opposite direction, given a convex structure $(X,t)$ on $V$, applying the above construction to $\overline V=\set{t\le T}$ for $T$ large enough shows that $V\cong\operatorname{int}\overline{V}$.)
\end{rmk}

\subsubsection{Symplectizations}
Given a (co-oriented, as always) contact manifold $(V,\xi)$, we will write $(\s V,\hat\omega)$  for its symplectization (see, e.g., \cite[\S1.3]{eliashberg-kim-polterovich-squeezing}), $\hat\lambda$ for its canonical Liouville form (so $\hat\omega=\dv\hat\lambda$) and $\pi=\pi^V\colon\s V\to V$ for the projection. Any choice of contact form $\alpha$ representing $\xi$ induces a diffeomorphism $\br{r,\pi}\colon\s V\to \R\times V$ such that $\hat\lambda=e^r\alpha$ and (hence) $\hat\omega=\dv(e^r\alpha)$ (where $\alpha=\pi^*\alpha$). When we have fixed such an $\alpha$, we will typically write $r$ for the corresponding coordinate, which we call a \defword{symplectization coordinate}. We also recall that the action of $\R$ on $\s V\cong\R\times V$ by translation is canonical -- i.e., independent of the chosen symplectization coordinate.

Given any contact embedding $i\colon(V_1,\xi_1)\to(V_2,\xi_2)$, there is an induced strict exact symplectic embedding $\s i\colon\s V_1\to\s V_2$ (``strict exact'' meaning $(\s i)^*\lambda_2=\lambda_1$). In symplectization coordinates induced by forms $\alpha_1$ and $\alpha_2$, it is given by $(\s i)(r,x)=(r-F(x),i(x))$, where $i^*\alpha_2=e^F\cdot\alpha_1$.

Recall that, given a contact form $\alpha$ representing $\xi$, any complex structure $J$ on $\xi$ which is \emph{compatible} with $\dv\alpha$ (in the usual sense that $\dv\alpha(-,J-)$ is symmetric and positive-definite) determines an $r$-invariant $\hat\omega$-compatible almost complex structure on $\s V$ which we will denote by $\hat{J}$, by requiring $\hat J\partial_r=\Rb_{\alpha}$ and $\rstr{\hat J}{\set{r_0}\times V}=J$.

We will sometimes also use $\hat{J}$ to denote an almost complex structure on a symplectic cobordism which is \emph{not} necessarily induced from another almost complex structure $J$ (see \S\ref{subsubsec:J-on-cobords}).

\section{Preparations for the definition of contact homology}\label{sec:preparations}
\subsection{Admissible contact forms}
We now introduce the class of contact forms we will be considering for the definition of contact homology. The purpose of the first three conditions in this definition is to give us control over the Reeb orbits, the fourth condition will allow us to prove compactness of the relevant moduli spaces of holomorphic curves, and the somewhat unnatural fifth condition is included for the sole purpose of being able to prove Lemma~\ref{lem:scaling-form} below.

Fixing a convex structure $(X,t)$ on $(V,\xi)$, the conditions below on the Reeb vector field $\Rb_\alpha$ involve a certain partition of $\Op_V(\infty)$ into three subsets. Namely, the partition $\Sigma=\Sigma^+\cup\Sigma^-\cup\Gamma$, actually comes from a partition $\Op_V(\infty)=\set{u>0}\cup\set{u<0}\cup\set{u=0}$, where $u=X\intprod\alpha$, in which $\set{u=0}$ is a smooth hypersurface. The conditions demand that $\Rb_\alpha$ point ``inward'' on $\set{u<0}$, ``outward'' on $\set{u>0}$, and point from $\set{u<0}$ to $\set{u>0}$ on $\set{u=0}$.

\begin{defn}\label{defn:admissible}
  Let $(X,t)$ be a convex structure for $V$. We say that a contact form $\alpha$ representing $\xi$ is \defword{$(X,t)$-admissible} if the following conditions hold for any sufficiently large $T>0$:
  \begin{enumerate}[(i)]
  \item\label{item:admissible-nondegen} $\alpha$ is non-degenerate (in the usual sense, as in \cite[\S1.2]{pardon-contact-homology-and-vfc}).
  \item\label{item:admissible-s} Letting $u=X\intprod\alpha$, we have $u\cdot(\Rb_{\alpha}\intprod\dt)>0$ on $\set{t>T}\cap\set{u\ne0}$.
  \item\label{item:admissible-t} With $u$ as above, $\Rb_\alpha$ is positively transverse to the hypersurface $\set{u=0}\cap\set{t>T}$ (``positively'' in the sense that $\Rb_{\alpha}\intprod\du>0$).
  \item\label{item:admissible-converge} For any $L>0$, the restrictions $\rstr{\alpha}{\set{t_0<t<t_0+L}}$, which can be viewed (for $t_0$ large enough) as contact forms on $\Sigma\times(0,L)$, converge in $\cc^\infty$ to a contact form on $\Sigma\times(0,L)$ as $t_0\to\infty$.
  \item\label{item:admissible-scale} For any smooth, positive, non-decreasing function $g(t)$, the above conditions \ref{item:admissible-s}~and~\ref{item:admissible-t} hold with $\alpha$ replaced by $g(t)\cdot\alpha$.
  \end{enumerate}
\end{defn}

\begin{rmk}\label{rmk:alpha-admiss-infty}
  Note that the condition of being $(X,t)$-admissible only depends on the restriction of $(X,t)$ to $\Op_V(\infty)$.
\end{rmk}

\begin{rmk}\label{rmk:admiss-rel-bounded}
  For any two $(X,t)$-admissible contact forms $\alpha$ and $\alpha'$, it follows from condition \ref{item:admissible-converge} that the function $\log(\frac{\alpha}{\alpha'})$ is bounded (i.e., $\frac{\alpha}{\alpha'}$ is bounded above and away from 0).
\end{rmk}

\begin{propn}\label{propn:admiss-reeb-orbits}
  If $\alpha$ satisfies conditions~\ref{item:admissible-s}~and~\ref{item:admissible-t} above for some $T>0$, then all of its Reeb orbits are contained in $\set{t<T}$.
\end{propn}
\begin{proof}
  Fix some such $T$, and suppose there were a Reeb orbit $\gamma$ which intersected $\set{t>T}$. Then $\gamma$ must intersect $\set{t>T}\cap\set{u>0}$ or $\set{t>T}\cap\set{u<0}$ by \ref{item:admissible-t}. But then $\gamma$ cannot be a closed orbit: for instance, once $\gamma$ enters $\set{t>T}\cap\set{u>0}$, then by \ref{item:admissible-t}, it cannot escape this region, and by \ref{item:admissible-s}, the function $t$ must be increasing along $\gamma$. The argument is the same in the other region, by considering $\gamma$ in reverse.
\end{proof}

\subsubsection{Existence of admissible contact forms}
We now want to show that for any convex structure $(X,t)$, there exist $(X,t)$-admissible contact forms $\alpha$. Note that besides non-degeneracy, all the conditions in the definition of admissibility only depend on the restriction of $\alpha$ to $\Op_V(\infty)$, and thus we begin by restricting our attention to this set.

We will proceed by first obtaining a form which is $t$-invariant and has a special form near the dividing set $\Gamma$, and then modifying it so as to satisfy the required conditions.

Until further notice, let $(V,\xi)$ be a contact manifold equipped with a diffeomorphism $\br{\pi,t}\colon V\to\Sigma\times\R$ such that $\xi$ is $t$-invariant (so that $\partial_t$ is a contact vector field and the slices $\set{t=t_0}$ are convex).

\begin{lem}\label{lem:t-invt-form}
  There exist forms $\alpha$ on $V$ representing $\xi$ which are $t$-invariant in the sense that $\partial_t\alpha=0$. Moreover:
  \begin{enumerate}[(i)]
  \item For any such $\alpha$, the restriction $\rstr{\alpha}{\Gamma}$ (where $\Gamma\subset\Sigma=\set{t=t_0}$ is the dividing set) is contact.
  \item\label{item:t-invt-form-beta-u} Any such $\alpha$ has the form $\beta+u\dt=\pi^*\beta+(\pi^*u)\dt$ for a uniquely determined $\beta\in\Omega^1(\Sigma)$ and $u\colon\Sigma\to\R$ (in fact, $u=\partial_t\intprod\alpha$).
  \end{enumerate}
\end{lem}
\begin{proof}
  See \cite[\S3B]{giroux-convexity-in-contact} or \cite[\S2.2]{colin-others-sutures-and-ch}.
\end{proof}

We next single out certain $t$-invariant $\alpha=\beta+u\dt$ (with $\beta$ and $u$ as in \ref{item:t-invt-form-beta-u} above) for which $\beta$ is ``cylindrical'' near $\Gamma$.

As in \cite[Lemma~2.2]{colin-others-sutures-and-ch} and \cite[3B]{giroux-convexity-in-contact}, we can consider the \emph{characteristic line field} $L_\Sigma=L_{\Sigma,\alpha}\defeq\ker(\rstr{\dv\alpha}{\xi\cap\Tn\Sigma})\subset\xi\cap\Tn\Sigma$ of $\Sigma$. The line field $L_\Sigma$ is transverse to $\Gamma$ (see \textit{loc. cit.}) and hence there exist vector fields $Y$ on a neighbourhood of $\Gamma$ which direct $L_\Sigma$.
\begin{lem}\label{lem:normal-form-near-gamma}\killenumspace
  \begin{enumerate}[(i)]
  \item\label{item:normal-form-near-gamma-beta0}
    In any tubular neighbourhood $\br{\pi,s}\colon U\toi\Gamma\times(-a,b)$ of $\Gamma$ in $\Sigma$ obtained by flowing along a vector field $Y$ directing $L_\Sigma$, we have that
    \[
      \beta=\lambda\beta_0=\lambda(\pi^*\beta_0)
    \]
    for some $\beta_0\in\Omega^1(\Gamma)$ and some positive function $\lambda$ on $U$.
  \item\label{item:normal-form-near-gamma-lambda-1} Moreover, there exists a choice of $\alpha$ and of vector field $Y$ directing $L_\Sigma$ so that $u=s$ and $\lambda=1$ on $U$, hence giving
    \[
      \rstr{\alpha}{U}=\beta_0+s\dt.
    \]
  \end{enumerate}
\end{lem}
\begin{proof}
  The statement and proof are very similar to those of \cite[Corollary~2.5]{colin-others-sutures-and-ch}; we leave the details to the reader.\qedhere
\end{proof}

Continuing with a fixed $t$-invariant form $\alpha$, assume now that $\alpha$ satisfies the conditions in the above lemma with respect to some $Y$, so $\rstr{\alpha}{U}=\beta_0+s\dt$. Let $\varepsilon>0$ be such that $(-\varepsilon,\varepsilon)$ is contained in the range of $s$.

Next, fix a function $\phi\colon(-\varepsilon,\varepsilon)\times\R\to\R_{>0}$ such that $\phi(s,t)=\abs{s}\I$ for $\abs{s}\ge\varepsilon/2$. Then the function
$\phi=\phi\circ\br{s,t}$ on $\set{\abs{s}<\varepsilon}\subset V$ agrees with $\abs{u}\I$ for $\abs{s}\ge\varepsilon/2$, and hence extends to a function $\Phi$ on $V$ agreeing with $\abs{u}\I$ outside of $\set{\abs{s}<\varepsilon/2}$.

We now set $\alpha_\phi=\Phi\alpha$, and we have:
\begin{lem}\label{lem:alpha-phi-reeb-field}
  The Reeb vector field of $\alpha_\phi$ is given by
    \[
      \Rb_{\alpha_\phi}=
      \begin{cases}
        \phi^{-2}\Big((\phi+s\phi_s)\Rb_{\beta_0}+\phi_t\partial_s-\phi_s\partial_t\Big)
        &\text{for }\abs{s}\le\varepsilon/2\\
        \pm\partial_t&\text{on }\set{\pm\partial_t\intprod\alpha>0}\setminus\set{\abs{s}<\varepsilon/2},
      \end{cases}
    \]
    where we write $\phi_s$ and $\phi_t$ for the partial derivatives of $\phi$.
\end{lem}
Here, $\Rb_{\beta_0}$ denotes the Reeb vector field of the form $\beta_0=\rstr{\alpha}{\Gamma}$, which is contact by Lemma~\ref{lem:t-invt-form}, and which we extend to $\set{\abs{s<\varepsilon}}\subset V$ by making it constant in $s$ and $t$.
\begin{proof}
  Firstly, on the complement of $\set{\abs{s}<\varepsilon/2}$, we have that $\alpha_\phi=\abs{u}\I\beta\pm\dt$, for which it is clear that the Reeb vector field is $\pm\partial_t$.

  On $\set{\abs{s}<\varepsilon}$, we have $\alpha_\phi=\phi\beta_0+\phi s\dt$, and hence
  \[
    \dv\alpha_\phi=
    \phi_s\ds\wedge\beta_0+\phi_t\dt\wedge\beta_0+\phi\dv\beta_0+
    (\phi+s\phi_s)\ds\wedge\dt
  \]
  We now make the \textit{ansatz} $\Rb_{\alpha_\phi}=a\Rb_{\beta_0}+b\partial_s+c\partial_t$ for some functions $a,b,c$. The requirement $\Rb_{\alpha_\phi}\intprod\dv\alpha_{\phi}$ then gives a (rank two) system of three linear equations in $a,b,c$, with the solution $a=H(\phi+s\phi_s)$, $b=H\phi_t$, and $c=-H\phi_s$ for some function $H$. The condition $\Rb_{\alpha_\phi}\intprod\alpha_\phi=1$ then gives $H=\phi^{-2}$.
\end{proof}

Now let $(V,\xi)$ be any contact manifold with a convex structure $(X,t)$.
\begin{propn}\label{propn:admissible-forms-exist}
  There exists an $(X,t)$-admissible form $\alpha$.
\end{propn}
\begin{proof}
  For the constant $T$ in Definition~\ref{defn:admissible}, we take any $T>0$ sufficiently large so that $X\intprod\dt=1$ on $\set{t>T}$. We first define $\alpha$ on $\set{t>T}\cong\Sigma\times(T,\infty)$.

  Namely, we take $\alpha_\phi$ as in the above lemma, where we additionally assume that $\phi$ satisfies the following conditions
  \begin{itemize}
  \item $s\phi_s<0$ when $s\ne0$.
  \item $\phi_t>0$ when $s=0$.
  \item $\phi$ is bounded and increasing in $t$, and the function $\phi_\infty(s)=\lim_{t\to\infty}\phi(s,t)$ is smooth.
  \end{itemize}
  (We can find such $\phi$ of the form $\phi(s,t)=\rho_1(s)+\rho_2(s)\rho_3(t)$.)

  Considering condition \ref{item:admissible-converge} of admissibility (Definition~\ref{defn:admissible}), we have that $\rstr{\alpha_\phi}{\set{t_0<t<t_0+L}}\to\alpha_{\phi_\infty}$ as $t_0\to\infty$.

  Next, it is clear from the description of $\Rb_{\alpha_\phi}$ that conditions \ref{item:admissible-s}~and~\ref{item:admissible-t} are also satisfied.

  Moreover, by Proposition~\ref{propn:admiss-reeb-orbits}, all of the Reeb orbits of any such extension will be contained in $\set{t\le T}$, so extending $\alpha$ to any non-degenerate form on $\set{t\le T}$ gives us a non-degenerate form on $V$ (here, we are using that non-degenerate forms are generic).

  It remains to verify condition~\ref{item:admissible-scale}. Let $g\colon(T,\infty)\to\R$ be positive, bounded, and non-decreasing.
  We want to show that $g\alpha_\phi$ still satisfies \ref{item:admissible-s}~and~\ref{item:admissible-t}.

  On $\set{\abs{s}<\varepsilon}$, this is clear, as we have $g\alpha_\phi=(g\phi)\alpha$, and $g\phi$ still satisfies the above conditions on $\phi$.

  Let us now compute $\Rb_{g\alpha_\phi}$ on the complement of $\set{\abs{s}\le\varepsilon/2}$, where we have $\alpha_\phi=\abs{u}\I\beta\pm\dt=\abs{u}\I(\beta+u\dt)$. Consider the (characteristic) vector field $Z$ on $\Sigma$ characterized by $Z\intprod\dv(\abs{u}\I\beta)=\abs{u}\I\beta$.

  Making the \emph{ansatz} $\Rb_{g\alpha_\phi}=a\partial_t+bZ$, we find that
  \[
    \Rb_{g\alpha_\phi}=\pm\plr{\frac1g\partial_t-\frac{g'}{g^2}Z}.
  \]

  Next, we consider $\Rb_{g\alpha_\phi}$ on $\set{\abs{s}<\varepsilon}$, where $\alpha_\phi=\phi(\beta_0+s\dt)$. Here, by the same computation as in Lemma~\ref{lem:alpha-phi-reeb-field}, we find that
  \[
    \Rb_{g\alpha_\phi}=
    (g\phi)^{-2}((g\phi+s(g\phi)_s)\Rb_{\beta_0}+(g\phi)_t\partial_s-(g\phi)_s\partial_t)=
    \frac{1}{g}\Rb_{\alpha_\phi}+\frac{g'}{g^2}\phi\I\partial_s,
  \]
  from which it is clear that $g\alpha_\phi$ satisfies conditions~\ref{item:admissible-s}~and~\ref{item:admissible-t}.
\end{proof}

\subsection{Admissible almost complex structures}
Let $(X,t)$ be a convex structure on $(V,\xi)$.

\begin{defn}
  Given a $(X,t)$-admissible form $\alpha$ on $V$, we say that a $\dv\alpha$-compatible almost complex structure $J$ on $\xi$ is \defword{admissible} if $\rstr{J}{\Op_V(\infty)}$ is $t$-invariant, i.e., $\Ldv_XJ=0$ on $\Op_V(\infty)$.

  We say that $(\alpha,J)$ is an \defword{$(X,t)$-admissible pair} if $\alpha$ and $J$ are, respectively, a contact form and an almost complex structure, both of which are $(X,t)$-admissible.
\end{defn}

It is clear that admissible almost complex structures exist, since there exist $\rstr{\dv\alpha}{\rstr{\Tn V}{\Sigma}}$-compatible almost complex structures on $\rstr{\xi}{\rstr{\Tn V}{\Sigma}}$.

\begin{rmk}\label{rmk:J-admiss-infty}
  As in Remark~\ref{rmk:alpha-admiss-infty}, note that the notion of $(X,t)$-admissibility only depends on the restriction of $(X,t)$ to $\Op_V(\infty)$.
\end{rmk}

\subsection{Symplectic cobordisms}
As mentioned in \S\ref{sec:recollections-on-ch}, the invariance of contact homology with respect to chosen auxiliary data, as well as more general functoriality, is obtained by considering symplectic cobordisms between contact manifolds.

Because our contact manifolds are non-compact, some care must be taken in defining cobordisms, and in this section, we introduce a suitable notion of symplectic cobordism and the various constructions and structures on them that we will need.

We note that we will actually only make use of cobordisms of the form $\R\times V$, but we formulate a slightly more general notion, as it is not substantially more complicated, and helps to clarify when we are really using that we are on $\R\times V$; also, the added generality may be useful in future work.

\subsubsection{Symplectic cobordisms}\label{subsubsec:symp-cobords}
Let $(V^\pm,\xi^\pm)$ be contact $(2n-1)$-manifolds with convex structures $(X^\pm,t^\pm)$, and suppose we are given a contactomorphism $i\colon\set{t^+>T}\toi\set{t^->T}$ for some $T\in\R$ preserving the convex structure, i.e., $i_*X^+=X^-$ and $t^-\circ i=t^+$.

\begin{defn}\label{defn:exact-symplec-cobord}
  An \defword{exact symplectic cobordism from} $(V^+,\xi^+,X^+,t^+)$ \defword{to} $(V^-,\xi^-,X^-,t^-)$ \defword{relative to} $i$ (or just \defword{symplectic cobordism from} $V^+$ \defword{to} $V^-$) is an exact symplectic $2n$-manifold $(\hat W,\dv\hat\lambda)$ equipped with strict exact symplectic embeddings $j_\pm\colon\s V^\pm\to\hat W$ such that the following conditions hold.
  \begin{enumerate}[(i)]
  \item The complement $\hat W\setminus(j_+(\s V^+)\cup j_-(\s V^-))$ of the images of $j_\pm$ is compact.
  \item The following diagram commutes.
    \begin{equation}\label{eq:symp-cob-diag}
      \begin{tikzcd}[row sep=10pt]
        \s\set{t^+>T}\ar[r, hook]\ar[dd, "S i"']&\s V^+\ar[rd, "j_+"]&\\
        &&\hat{W}\\
        \s\set{t^->T}\ar[r, hook]&\s V^-\ar[ru, "j_-"']
      \end{tikzcd}
    \end{equation}
  \item\label{item:exact-symplec-cobord-disjoint-ends} With respect to some symplectization coordinates $\s V^\pm\toi\R\times V^\pm$ (coming from $(X^\pm,t^\pm)$-admissible contact forms $\alpha^\pm$), there exists an $N\in\R$ such that the embeddings $\rstr{j_+}{[N,\infty)\times V^+}$ and $\rstr{j_-}{(-\infty,-N]\times V^-}$ are proper and have disjoint images.
  \end{enumerate}
\end{defn}

\begin{rmk}
  It follows from Remark~\ref{rmk:admiss-rel-bounded} that whether condition \ref{item:exact-symplec-cobord-disjoint-ends} holds does not depend of the particular $(X^\pm,t^\pm)$-admissible contact forms $\alpha^\pm$ used to define the symplectization coordinates.
\end{rmk}

Note that there is a canonical symplectic cobordism from $V$ to itself for any $V$, namely with $i=\id_V$ and $j_\pm=\id_{\s V}$.

In a symplectic cobordism $\hat{W}$, we have a ``horizontal'' end $j_\pm(\s\Op_{V^\pm}(\infty))$, which we will denote by $\Op_{\hat W}(\infty)$ (or $\Op(\infty)$), as well as positive and negative ``vertical'' ends, which are the images of $\rstr{j_+}{(N,\infty)\times V}$ and $\rstr{j_-}{(-\infty,N)\times V}$ with $N$ large as above (and again with symplectization coordinates $\s V^\pm\toi\R\times V^\pm$ coming from $(X^\pm,t^\pm)$-admissible contact forms), which we will denote by $\Op_{\hat W}(\pm\infty)$ (or $\Op(\pm\infty)$).

We also have a function $t$ on $\Op_{\hat W}(\infty)$ such that $t\circ j_\pm=t^\pm\circ\pi^{V^\pm}$. We may implicitly (arbitrarily) extend $t$ to a function on $\hat{W}$, bounded on $\hat{W}\setminus\Op_{\hat{W}}(\infty)$.

Note that the flow of $X^\pm$ on $\Op_{V^\pm}(\infty)$ is by contactomorphisms, and hence lifts to a flow on $\Op_{\s V^\pm}(\infty)$, and hence on $\Op_{\hat W}(\infty)$, by strict exact symplectomorphisms (given by a vector field $\hat X$ on $\Op_{\hat W}(\infty)$ with $\hat X\intprod\dt=1$). In particular, the hypersurfaces $\set{t=t_0}\subset\hat{W}$ for various $t_0$ (large enough) are all canonically isomorphic (where the isomorphisms cover the isomorphisms between the various $\Sigma_{t_0}\subset V^\pm$).

\subsubsection{Almost complex structures on cobordisms}\label{subsubsec:J-on-cobords}
In the case of a trivial cobordism $\s V$ with $(V,\xi)$ convex, we already have a notion of admissible complex structure; namely, it is one of the form $\hat J$ for an admissible almost complex structure $J$ on $\xi$.

We can generalize this to arbitrary convex cobordisms as follows:
\begin{defn}
  Let $(V^+,\xi^+)$ and $(V^-,\xi^-)$ be convex contact manifolds and let $J^\pm$ be admissible almost complex structures on $\xi^\pm$.
  An almost complex structure $\hat J$ on a symplectic cobordism $(\hat W,\hat\omega)$ from $V^+$ to $V^-$ is \defword{admissible} with respect to $J^\pm$ if it satisfies the following conditions:
  \begin{enumerate}[(i)]
  \item $\hat J$ is $\hat{\omega}$-compatible
  \item There exist neighbourhoods $\Op_{\hat{W}}(\pm\infty)$ of infinity on which $(j_\pm)_*\wh{J^{\pm}}$ and $\hat J$ coincide.
  \item For any $L>0$, the almost complex structures $\rstr{\hat{J}}{\set{t_0<t<t_0+L}}$ converge in $\cc^\infty_{\loc}$ as $t_0\to\infty$ to some almost complex structure (where we are identifying the various $\set{t_0<t<t_0+L}\subset\hat{W}$ as explained at the end of \S\ref{subsubsec:symp-cobords}).
  \end{enumerate}
\end{defn}

\begin{lem} \label{lem:acs-trivial-cobordism}
  Let $(V,\xi)$ be a convex contact manifold. For any pair of admissible almost complex structures $J^\pm$ on $\xi$, there exists an almost complex structure $\hat J$ on the symplectization $\s V$ (viewed as a cobordism from $V$ to itself) which is admissible with respect to $J^\pm$.
\end{lem}
\begin{proof}
	Any convex contact manifold $(V,\xi)$ is the interior of a compact contact manifold $(\overline{V},\overline{\xi})$ with convex boundary. Moreover, any admissible almost complex structure $J$ on $\xi$ can be canonically extended to a compatible almost complex structure $\bar J$ on $\overline\xi$. It is a standard fact that any compatible almost complex structure defined on a closed subset of a symplectic manifold can be extended to the whole manifold. Hence, there exists a compatible almost complex structure $\hat J$ on $\s \overline{V}$ which agrees with $\bar{J}^\pm$ on the ends. The restriction of $\hat J$ to $\s V$ is the desired admissible almost complex structure.
\end{proof}

\subsubsection{Families of cobordisms}
\begin{defn}
  Given $(V^\pm,\xi^\pm,X^\pm,t^\pm)$ and $i$ as in Definition~\ref{defn:exact-symplec-cobord}, a \defword{family of symplectic cobordisms} from $V^+$ to $V^-$ relative to $i$ is a manifold $\hat{W}$ with a family of symplectic forms $(\hat\omega^\tau)_{\tau\in P}$ (with $P$ some parameter space, which will for us always be an interval), constant outside of a compact subset of $\hat{W}$, together with maps $j_\pm$ as in \S\ref{subsubsec:symp-cobords}, such that each $(\hat{W},\omega^\tau,j_\pm)$ is a symplectic cobordism from $V^+$ to $V^-$ relative to $i$.
\end{defn}

\begin{defn}
  Let $(\hat{W}, (\hat\omega^\tau)_{\tau\in P})$ be a family of symplectic cobordisms as in the above definition and let $J^\pm$ be admissible complex structures on $\xi^\pm$. A family $(\hat J^\tau)_{\tau \in P}$ of almost complex structures on $\hat{W}$ is \defword{admissible} with respect to $J^\pm$ if for each $\tau \in P$, $\hat J^\tau$ is admissible with respect to $J^\pm$ as an almost complex structure on $(\hat{W}, \hat \omega^\tau)$.
\end{defn}

\begin{lem} \label{lem:acs-family-cobordisms}
	Let $(\hat{W}, (\hat\omega^\tau)_{\tau\in [0,1]})$ be a family of symplectic cobordisms, let $J^\pm$ be admissible complex structures on $\xi^\pm$, and let $\hat J^0$ (resp. $\hat J^1$) be an admissible almost complex structure on $(\hat{W}, \hat\omega^0)$ (resp. $(\hat{W},\hat\omega^1)$). Then there exists a family of almost complex structures $(\hat J^\tau)_{\tau \in [0,1]}$ on $\hat W$ extending $(\hat{J}^0,\hat{J}^1)$ which is admissible with respect to $J^\pm$.
\end{lem}
\begin{proof}
	The proof is essentially the same as that of Lemma~\ref{lem:acs-trivial-cobordism}. Indeed, the family $(\hat{W}, (\hat\omega^\tau)_{\tau\in [0,1]})$ can be ``horizontally compactified'' to a family of cobordisms from $\overline{V}^+$ to $\overline{V}^-$, which reduces the problem to extending a compatible almost complex structure defined on a closed subset of a symplectic manifold.
\end{proof}

\subsubsection{Gluings of cobordisms}
Now suppose we are given three contact manifolds with convex structures $(V^k,\xi^k,X^k,t^k)$ for $k=0,1,2$, convex-structure-preserving contactomorphisms $i^{\ell(\ell+1)}\colon\Op_{V^\ell}(\infty)\to\Op_{V^{\ell+1}}(\infty)$ ($\ell=0,1$), as well as symplectic cobordisms $(\hat W^{\ell(\ell+1)},\dv\lambda^{\ell(\ell+1)})$ from $V^\ell$ to $V^{\ell+1}$ relative to $i^{\ell(\ell+1)}$, and denote the associated maps by $j_\pm^{\ell(\ell+1)}$.

We then obtain, for each $\tau\in[0,\infty)$ a symplectic cobordism $\hat W^{02,\tau}$ from $V^0$ to $V^2$ relative to $i^{12}\circ i^{01}$ by forming the pushout
\begin{equation}\label{eq:tau-gluing-pushout}
  \begin{tikzcd}
    \s V^1\ar[r, "\mu_{-\tau/2}"]\ar[d, "\mu_{\tau/2}"']&
    \s V^1\ar[r, "j^{01}_-"]&
    \hat W^{01}\ar[dd, "\inc^{01}"]\\
    \s V^1\ar[d, "j^{12}_+"']\\
    \hat W^{12}\ar[rr, "\inc^{12}"]&&\hat W^{02,\tau}
    \ar[lluu, "\ulcorner", phantom, pos=0]
  \end{tikzcd}
\end{equation}
where $\mu_\tau$ denotes translation by $\tau$ with respect to the $\R$ action on $\s V^1$.

We equip $\hat{W}^{02,\tau}$ with the exact symplectic structure $\hat\omega^{02}=\dv\hat\lambda^{02}$ characterized as follows. Since $\mu_\tau^*\hat\lambda=e^\tau\hat\lambda$ and since $j_\pm^{\ell(\ell+1)}$ are strict exact, there is a unique form $\hat\lambda^{02}$ on $\hat W^{02,\tau}$ with $(\inc^{01})^*\hat\lambda^{02}=e^{\tau/2}\hat\lambda^{01}$ and $(\inc^{12})^*\hat\lambda^{02}=e^{-\tau/2}\hat\lambda^{12}$.

For the associated exact symplectic embeddings $j^{02,\tau}_+\colon\s V^0\to \hat W^{02,\tau}$ and $j^{02,\tau}_-\colon\s V^2\to \hat W^{02,\tau}$, we have no choice but to take
\[
  \s V^0\tox{\mu_{-\tau/2}}\s V^0\tox{j_+^{01}}\hat W^{01}
  \tox{\inc^{01}}\hat W^{02,\tau}
  \quad\text{and}\quad
  \s V^2\tox{\mu_{\tau/2}}\s V^2\tox{j_+^{12}}\hat W^{12}
  \tox{\inc^{12}}\hat W^{02,\tau}
\]
respectively.

It is then easy to see that the diagram \eqref{eq:symp-cob-diag} commutes and that the remaining conditions in the definition of symplectic cobordism obtain.

\begin{defn}
  For each $\tau\in[0,\infty)$, we call the exact symplectic cobordism $\hat W^{02,\tau}$ described above the \defword{$\tau$-gluing} of $\hat{W}^{01}$ and $\hat{W}^{12}$.
\end{defn}

Note that if $\hat{W}^{01}$ and $\hat{W}^{12}$ are both the trivial cobordism $\s V$, then the $\tau$-gluing is again canonically isomorphic to the trivial cobordism $\s V$, with $\inc^{01}$ and $\inc^{12}$ corresponding to $\mu_{\tau/2}$ and $\mu_{-\tau/2}$, respectively.

\subsubsection{Almost complex structures on $\tau$-gluings}
With $\hat{W}^{01}$ and $\hat{W}^{12}$ as above, suppose that we have admissible complex structures $J^k$ on $\xi^k$ ($k=0,1,2$), and admissible almost complex structures $\hat{J}^{01}$ and $\hat{J}^{12}$ on $\hat{W}^{01}$ and $\hat{W}^{12}$ with respect to these.

Consider the diagram \eqref{eq:tau-gluing-pushout}. Note that, by definition of admissibility, the almost complex structures on $\hat{W}^{01}$ and $\hat{W}^{12}$ agree on $\s V^1$ for $\tau$ sufficiently large, and hence induce an (admissible!) almost complex structure $\hat{J}^{02,\tau}$ on $\hat{W}^{02,\tau}$. Hence, in this situation, we will often tacitly assume $\hat{W}^{02,\tau}$ to be equipped with this almost complex structure.

For small $\tau$, we will only obtain an almost complex structure $\hat{J}^{02,\tau}$ on the positive and negative ends of $\hat{W}^{02,\tau}$, and we arbitrarily extend this to an admissible almost complex structure on all of $\hat{W}^{02,\tau}$.

\subsection{Compactness}\label{subsec:compactness}
We now come to the compactness theorems for moduli spaces of holomorphic curves in symplectic cobordisms between convex contact manifolds that we will need in order to define contact homology. As our definition of the latter will precisely follow the construction in \cite{pardon-contact-homology-and-vfc}, we need only ensure that the compactness theorems used there (taken from \cite{compactness-in-sft}), which are defined for compact $V$, go through in our context.

However, those theorems go through as long as all of the holomorphic curves in the moduli space under consideration are confined to a compact subset of $V$ (see \cite[Theorem~11.1]{compactness-in-sft}). More precisely, when considering holomorphic curves in $\s V$, they should be contained in $(\pi^V)\I(K)$ for some compact $K\subset V$, and when considering curves in a general symplectic cobordism, they should be confined to the complement of some neighbourhood $\Op(\infty)$ of horizontal infinity.

Hence, in this section, we provide various results constraining holomorphic curves in symplectic cobordisms to certain subsets. To do this, we will take advantage of the (asymptotic) $t$-invariance of the almost complex structure we are considering, and appeal to the ``target-local'' version of Gromov compactness from \cite{fish-target-local}.

\subsubsection{Holomorphic curves in a $\tau$-glued cobordism}
We begin with the most complicated of the situations we will be considering, as the proof in this case includes all of the ideas needed in the remaining proofs.

To begin with, we will want to establish the existence of symplectic forms in cobordisms with certain desirable properties (\textit{cf}. \cite[\S2.10]{pardon-contact-homology-and-vfc}).
\begin{lem}\label{lem:interpolating-form}
  Let $(\hat{W},\dv\hat\lambda)$ be a symplectic cobordism from $(V^+,\xi^+,X^+,t^+)$ to $(V^-,\xi^-,X^-,t^-)$ and fix admissible pairs forms $(\alpha^\pm,J^\pm)$ on $V^\pm$. Then for any admissible almost complex structure $\hat J$ on $\hat W$, there exists an exact symplectic form $\tilde{\omega}=\dv\tilde\lambda$ on $\hat{W}$ with the following properties:
  \begin{enumerate}[(i)]
  \item\label{item:interpolating-compatible} $\hat J$ is $\tilde{\omega}$-compatible.
  \item\label{item:interpolating-vert} Consider the trivialization $(r_\pm,\pi)\colon\s V^\pm\toi\R\times V^\pm$ induced by $\alpha^\pm$, and let $r_\pm\defeq r_\pm\circ j_\pm\I\colon\Op_{\hat W}(\pm\infty)\to\R$:

    We then have $\tilde\lambda=e^{R_\pm(r_\pm)}\alpha^\pm$ on $\Op_{\hat{W}}(\pm\infty)$ for some bounded, increasing functions $R_\pm\colon\R\to\R$.
  \item\label{item:interpolating-horiz} For each $L>0$, the restrictions $\rstr{\tilde\lambda}{\set{t_0<t<t_0+L}}$ converge in $\cc^{\infty}$ as $t_0\to\infty$ to some primitive of a symplectic form (where we are again using the identifications between the various $\set{t=t_0}$ as explained at the end of \S\ref{subsubsec:symp-cobords}).
    (Note that on $\Op_{\hat W}(\pm\infty)$, this follows from \ref{item:interpolating-vert} and condition \ref{item:admissible-converge} of admissibility, Definition~\ref{defn:admissible}).
  \end{enumerate}
\end{lem}
\begin{proof}
  Since $\hat\lambda$ agrees with $e^{r_\pm}\alpha^\pm$ on $\Op_{\hat{W}}(\pm\infty)$, we can simply replace it on this region by $\tilde\lambda=e^{R_\pm(r_\pm)}\alpha^\pm$, where $R_+(r)$ is any bounded above increasing function which is equal, for sufficiently large $N>0$, to $r$ on $(-\infty,N]$, and $R_-(r)$ is any bounded below increasing function agreeing with $r$ on $[N,\infty)$. Property~\ref{item:interpolating-vert} then follows by construction. Let us moreover choose $N$ large enough so that $j_\pm^*\hat J=\wh{J^\pm}$ on $\set{\pm r_\pm\ge N}$ (such an $N$ exists by admissibility of $\hat J$).

  Property~\ref{item:interpolating-horiz} follows from the corresponding properties of $\alpha^\pm$ and $\hat\lambda$.

  Property~\ref{item:interpolating-compatible} is clear where $\tilde{\lambda}=\hat\lambda$. Elsewhere, where $\tilde\lambda=e^{R_\pm(r_\pm)}\alpha^\pm$, we have by our choice of $N$ that $j_\pm^*\hat J=\wh{J^\pm}$, and the same proof that shows that $\wh{J^\pm}$ is $\dv(e^{r\pm}\alpha^\pm)$-compatible also shows that it is $\dv(e^{R_\pm(r_\pm)})$-compatible.
\end{proof}

\begin{lem}\label{lem:interpolating-form-glued}
  Let $(\hat{W}^{\ell(\ell+1)},\dv\hat\lambda^{\ell(\ell+1)})$ (for $\ell=0,1$) be symplectic cobordisms from $(V^\ell,\xi^\ell,X^\ell,t^\ell)$ to $(V^{\ell+1},\xi^{\ell+1},X^{\ell+1},t^{\ell+1})$ and fix admissible pairs $(\alpha^k,J^k)$ (for $k=0,1,2$) on $V^k$.

  Moreover, fix admissible almost complex structures $\hat{J}^{\ell(\ell+1)}$ on $\hat{W}^{\ell(\ell+1)}$ and consider forms $\tilde\lambda^{\ell(\ell+1)}$ on $\hat{W}^{\ell(\ell+1)}$ as in Lemma~\ref{lem:interpolating-form}.

  Also, fix a smooth bounded increasing function $\rho(r)$, and assume that
  \[
    \lim_{r\to-\infty}R^{01}_-(r)>
    \lim_{r\to\infty}\rho(r)
    \quad\text{and}\quad
    \lim_{r\to-\infty}\rho(r)>
    \lim_{r\to\infty}R^{12}_+(r)
  \]
  where $R^{01}_-$ and $R^{12}_+$ are functions as in Lemma~\ref{lem:interpolating-form}~\ref{item:interpolating-vert}.

  Then there exists a family $\set{\tilde\omega^\tau=\dv\tilde\lambda^\tau}_{\tau\in(T,\infty)}$ (for $T>0$ large enough) of exact symplectic forms on the $\tau$-glued cobordisms $\hat{W}^{02,\tau}$ with the following properties:
  \begin{enumerate}[(i)]
  \item\label{item:interpolating-glued-agree-cob} The forms $(\inc^{\ell(\ell+1)})^*\tilde\lambda^\tau$ on $\hat{W}^{\ell(\ell+1)}$ converge in $\cc^\infty_\loc$ to $\tilde\lambda^{\ell(\ell+1)}$ as $\tau\to\infty$.
  \item\label{item:interpolating-glued-agree-mid} The forms $j_1^*\tilde\lambda^\tau$ on $\s V^1$ converge in $\cc^\infty_\loc$ to $e^{\rho(r)}\alpha$ as $\tau\to\infty$, where $j_1$ is the map
  \[
    \inc^{01}\circ j^{01}_-\circ\mu_{-\tau/2}=
    \inc^{12}\circ j^{12}_+\circ\mu_{\tau/2}
    \colon\s V^1\hto \hat W^{02,\tau}
  \]
  and $r$ is the symplectization coordinate on $\s V^1$ induced by $\alpha^1$.
  \item\label{item:interpolating-glued-compatible} The induced almost complex structure $\hat J^{02,\tau}$ on $\hat W^{02,\tau}$ is $\tilde{\omega}^\tau$-compatible.
  \item\label{item:interpolating-glued-horiz} As in Lemma~\ref{lem:interpolating-form}~\ref{item:interpolating-horiz}, the restrictions $\rstr{\tilde\lambda^\tau}{\set{t_0<t<t_0+L}}$ converge for any $L>0$ to a primitive of a symplectic form as $t_0\to\infty$.
  \end{enumerate}
\end{lem}
\begin{proof}
  Let us identify $\s V^1$ with $\R\times V^1$ via the isomorphism induced by $\alpha^1$. In particular, we have the map $j_1\colon\R\times V^1\to W^{02,\tau}$, as well as the maps $j_-^{01}\colon\R\times V^1\to\hat W^{01}$ and $j_+^{12}\colon\R\times V^1\to\hat W^{12}$.

  Since $\hat J^{01}$ and $\hat{J}^{12}$ are admissible, there is some $C>0$ such that $(j_-^{01})^*\hat J^{01}=\wh{J^1}$ on $(-\infty,-C)$ and $(j_+^{12})^*\hat J^{12}=\wh{J^1}$ on $(C,\infty)$. It follows (by the definition of $\hat{J}^{02,\tau}$) that $j_1^*\hat{J}^{02,\tau}=\wh{J^1}$ on $(-\tau/2+C,\tau/2-C)\times V^1$. Let us set $\sigma\defeq\tau/2-C$.

  We take $\tilde{\lambda}^\tau$ to be equal to $(\inc^{01})_*\tilde\lambda^{01}$ or $(\inc^{12})_*\tilde\lambda^{12}$ outside of $j_1((-\sigma/2,\sigma/2)\times V^1)$. It remains to define $\tilde{\lambda}^\tau$ on $j_1((-\sigma/2,\sigma/2)\times V^1)$.

  Next, on $j_1([-\sigma/4,\sigma/4]\times V^1)$, we define $\tilde\lambda^\tau$ to be $(j_1)_*(e^\rho\alpha^1)$.

  We now have that, on $j_1(((-\sigma,-\sigma/2]\cup{}[-\sigma/4,\sigma/4]\cup{}[\sigma/2,\sigma))\times V^1)$, $\tilde\lambda$ is equal to $e^{F(r)}\alpha^1$ for some function $F$ which is strictly increasing, and hence admits an extension to the whole interval $(-\sigma/2,\sigma/2)$ which is still strictly increasing.

  Property~\ref{item:interpolating-glued-horiz} follows from the corresponding properties on $\hat{W}^{\ell(\ell+1)}$. Property~\ref{item:interpolating-glued-compatible} is proven in the same way as in Lemma~\ref{lem:interpolating-form}.

  Finally, the convergence $\cc^\infty_\loc$-convergence in \ref{item:interpolating-glued-agree-cob}~and~\ref{item:interpolating-glued-agree-mid} is trivial since on any compact subset, the family of forms is eventually of constant.
\end{proof}

We now come to the main result of this section. We refer to \cite{compactness-in-sft,pardon-contact-homology-and-vfc} for the notion of \emph{holomorphic curve (or holomorphic building) with punctures asymptotic to Reeb orbits} in a symplectic cobordism.

As we will only be using a certain restricted class of such holomorphic curves and buildings, we make the following definition so as to avoid repeating the assumptions every time:
\begin{defn}
  Given a symplectic cobordism $\hat W$ from $(V^+,\xi^+)$ to $(V^-,\xi^-)$, representing contact forms $\alpha^\pm$ on $V^\pm$, and an almost complex structure $\hat J$ on $\hat W$, we say that a punctured $\hat J$-holomorphic curve or building $u$ in $\hat W$ is \defword{relevant} if it has genus 0, all of its negative punctures are asymptotic to Reeb orbits of $\alpha^-$ and it has a single positive puncture asymptotic to a Reeb orbit of $\alpha^+$, which we call its \defword{positive end}.
\end{defn}

\begin{propn}\label{propn:compactness-glued}
  Let $(\hat W^{\ell(\ell+1)},\hat\omega^{\ell(\ell+1)})$ ($\ell=0,1$) be symplectic cobordisms from $(V^\ell,\xi^\ell,X^\ell,t^\ell)$ to $(V^{\ell+1},\xi^{\ell+1},X^{\ell+1},t^{\ell+1})$.

  Fix admissible pairs $(\alpha^k,J^k)$ on $V^k$ (for $k=0,1,2$) and admissible almost complex structures $\hat J^{\ell(\ell+1)}$ on $\hat W^{\ell(\ell+1)}$ with respect to these, so that we have an induced admissible complex structure $\hat J^{02,\tau}$ on $\hat W^{02,\tau}$ for each $\tau\in{}[0,\infty)$.

  Also fix a constant $a>0$.

  Then there exists $T>0$ such that, for any $\tau\in{}[0,\infty)$, any relevant $\hat{J}^{02,\tau}$-holomorphic curve in $\hat W^{02,\tau}$ with positive end a Reeb orbit of $\alpha^0$ of action $<a$ is contained within $\set{t<T}\subset\hat W^{02,\tau}$.
\end{propn}
\begin{proof}
  For the sake of contradiction, suppose that there were a sequence $T_i\tox{i\to\infty}\infty$, a sequence $\tau_i\in{}[0,\infty)$, and a sequence of $\hat J^{02,\tau_i}$-holomorphic curves $v_i$ in $\hat W^{02,\tau_i}$ of the kind specified in the theorem, such that $v_i$ intersects $\set{t\ge T_i}$. Note that since all the Reeb orbits of $\alpha^0$ are contained in some $\set{t\le T}\subset V^0$, each $v_i$ must also intersect $\set{t\le T}\subset\hat{W}^{02,\tau}$.

  Let $\tilde{\omega}^{\ell(\ell+1)}=\dv\tilde{\lambda}^{\ell(\ell+1)}$ (for $\ell=0,1$) be forms on $\hat{W}^{\ell(\ell+1)}$ and $\tilde{\omega}^{\tau}=\dv\tilde{\lambda}^{\tau}$ be forms on $\hat{W}^{02,\tau}$ as in Lemma~\ref{lem:interpolating-form-glued}.

  Fix some small $\varepsilon>0$, and for each $n\in\Z\cap(T',\infty)$, consider the subset $C_{n,\tau}=\set{n-\varepsilon<t<n+1+\varepsilon)}$ of $\hat{W}^{02,\tau}$.

  Since by assumption $\tilde{\lambda}^{01}$ converges on $\Op_{\hat W^{01}}(+\infty)$ to a multiple $C\alpha^0$ of $\alpha^0$, we have by Stokes' theorem that $\int v_i^*\tilde{\omega}^{\tau_i}<Ca$. Hence, there must be some sequence of indices $n_i\tox{i\to\infty}\infty$ (where we can assume $T'+3<n_i<T_i-3$) such that $\int_{S_i}v_i^*\tilde{\omega}^{\tau_i}\tox{i\to\infty}0$, where $S_i\defeq v_i\I(C_{n_i,\tau_i})$.

  Each $C_{n,\tau}$ is the union of five regions $C_{n,\tau}^k=C_{n\tau}\cap(\hat{W}^{02,\tau})^k$ ($k=0,1,2$) and $C_{n,\tau}^{\ell(\ell+1)}=C_{n,\tau}\cap(\hat{W}^{02,\tau})^{\ell(\ell+1)}$ ($\ell=0,1$), where the $(\hat{W}^{02,\tau})^k$ and $(\hat{W}^{02,\tau})^{\ell(\ell+1)}$ are defined as follows.

  We define $(\hat W^{02,\tau})^k$ for $k=0,2$ to be neighbourhoods of infinity $\Op_{\hat W^{02,\tau}}(+\infty)$ and $\Op_{\hat W^{02,\tau}}(-\infty)$ on which $\hat J^{02,\tau}$ is equal to $(j_+^{02,\tau})_*\wh{J^{0}}$ and $(j_-^{02,\tau})_*\wh{J^{2}}$, respectively (more precisely, we take them to be the images under $\inc^{01}$ and $\inc^{12}$ of some fixed $\Op_{\hat W^{01}}(+\infty)$ and $\Op_{\hat W^{12}}(-\infty)$).

  We similarly take $(\hat{W}^{02,\tau})^1$ to be the image under the embedding
  \[
    j_1=
    \inc^{01}\circ j^{01}_-\circ\mu_{-\tau/2}=
    \inc^{12}\circ j^{12}_+\circ\mu_{\tau/2}
    \colon\s V^1\hto \hat W^{02,\tau}
  \]
  of the largest ``interval'' on which $j_1^*(\hat J^{02,\tau})$ is equal to $\wh{J^{1}}$, where interval means a subset of the form $(a,b)\times V^1$ (with respect to the symplectization parameter induced by $\alpha^1$).

  Finally, we $(\hat{W}^{02,\tau})^{01}$ to be the union of $(\hat{W}^{02,\tau})^1$ and everything ``above'' it, and $(\hat{W}^{02,\tau})^{12}$ to be the union of $(\hat{W}^{02,\tau})^1$ and everything ``below'' it (see Figure~\ref{fig:regions}).
  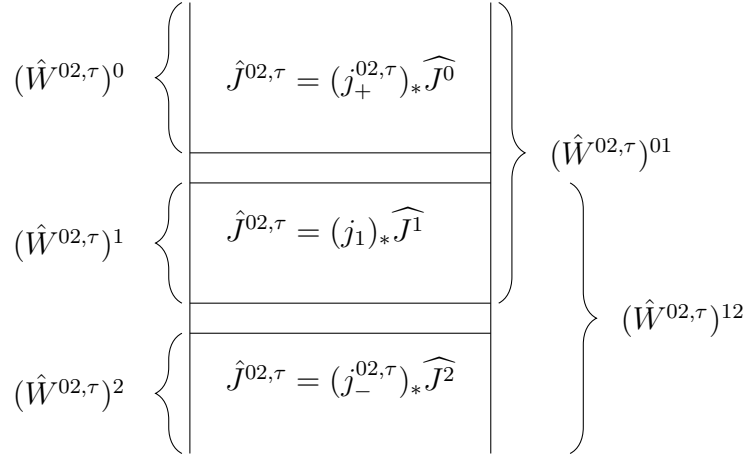
\begin{figure}
    \centering
    \begin{tikzpicture}[scale=2,y=-1cm,
    brace/.style={decorate,decoration={brace,amplitude=10pt,mirror},xshift=-1.5pt,yshift=0pt},
    rbrace/.style={decorate,decoration={brace,amplitude=10pt},xshift=1.5pt,yshift=0pt},
    sbrace/.style={decorate,decoration={brace,amplitude=10pt},xshift=15pt,yshift=0pt}]
    \draw (0,1)   -- (2,1);
    \draw (0,1.2) -- (2,1.2);
    \draw (0,2)   -- (2,2);
    \draw (0,2.2) -- (2,2.2);
    \draw (0,0)   -- (0,3);
    \draw (2,0)   -- (2,3);
    \node at (1,0.5) [align=center] {$\hat{J}^{02,\tau}=(j_+^{02,\tau})_*\wh{J^{0}}$};
    \node at (1,1.5) [align=center] {$\hat{J}^{02,\tau}=(j_1)_*\wh{J^{1}}$\ \ \ };
    \node at (1,2.5) [align=center] {$\hat{J}^{02,\tau}=(j_-^{02,\tau})_*\wh{J^{2}}$};
    \draw [brace] (0,0) -- (0,1) node [midway,xshift=-1.5cm] {$(\hat{W}^{02,\tau})^{0}$};
    \draw [brace] (0,1.2) -- (0,2) node [midway,xshift=-1.5cm] {$(\hat{W}^{02,\tau})^{1}$};
    \draw [brace] (0,2.2) -- (0,3) node [midway,xshift=-1.5cm] {$(\hat{W}^{02,\tau})^{2}$};
    \draw [rbrace] (2,0) -- (2,2) node [midway,xshift=1.5cm] {$(\hat{W}^{02,\tau})^{01}$};
    \draw [sbrace] (2,1.2) -- (2,3) node [midway,xshift=1.5cm] {$(\hat{W}^{02,\tau})^{12}$};
  \end{tikzpicture}
  \caption{The regions $(\hat{W}^{02,\tau})^k$ and $(\hat{W}^{02,\tau})^{\ell(\ell+1)}$.}\label{fig:regions}
\end{figure}

  Next, we have obvious embeddings $C_{n,\tau}^k\to D_n^k\defeq\set{n-\varepsilon<t<n+\varepsilon)}\subset\s V^k$ and $C_{n,\tau}^{\ell(\ell+1)}\to D_n^{\ell(\ell+1)}\defeq\set{n-\varepsilon<t<n+1+\varepsilon}\subset\hat W^{\ell(\ell+1)}$.

  Note that each $D_n^k$ and each $D_n^{\ell(\ell+1)}$ is canonically diffeomorphic, respectively, to $D^k\defeq\set{N-\varepsilon<t<N+1+\varepsilon}\subset\s V^k$ and $D^{\ell(\ell+1)}\defeq\set{N-\varepsilon<t<N+1+\varepsilon}\subset\hat W^{\ell(\ell+1)}$ (for some large fixed $N$). By pushing forward $\tilde{\omega}^{\tau_i}$ and $\hat{J}^{02,\tau_i}$ under the resulting embedding $C_{n_i,\tau_i}^k\to D^k$, we obtain a sequence $(\check{\omega}_i^k,\check{J}_i^k)$ of almost Hermitian structures on the images of these embeddings (and similarly for $D^{\ell(\ell+1)}$).
  Moreover, by the admissibility of $\check{J}_i^k$ and by Lemma~\ref{lem:interpolating-form-glued}, these converge in $\cc^\infty_\loc$ as $i\to\infty$ to an almost Hermitian structure $(\check{\omega}^k,\check{J}^k)=(\dv\check{\lambda}^k,\check{J}^k)$ on $D^k\cong\R\times \Sigma\times(N-\varepsilon,N+1+\varepsilon)$ (and $(\check{\omega}^{\ell(\ell+1)},\check{J}^{\ell(\ell+1)})=(\dv\check{\lambda}^{\ell(\ell+1)},\check{J}^{\ell(\ell+1)})$ on $D^{\ell(\ell+1)}$), where $\lambda^k=\dv(e^{R^k}\alpha_\infty)$ for some smooth bounded function $R^k\colon\R\to\R$, where $\alpha_\infty$ is a contact form on $\Sigma\times(N-\varepsilon,N+\varepsilon)$, and where $\check{J}^k=\wh{J_\infty}$ is the almost-complex structure induced from an $\alpha_\infty$-compatible complex structure on $\ker(\alpha_\infty)$.

  Hence, by restricting the maps $v_i$ and then composing with the embeddings $C^k_{n_i,\tau_i}\to D^k$ and $C^{\ell(\ell+1)}_{n_i,\tau_i}\to D^{\ell(\ell+1)}$, we obtain (respectively) $\check J^{k}_i$-holomorphic and $\check J_i^{\ell(\ell+1)}$-holomorphic curves $v_i^k$ and $v_i^{\ell(\ell+1)}$ in $D^k$ and $D^{\ell(\ell+1)}$.

  We will now prove that for one of the five sequences $v_i^k$ and $v_i^{\ell(\ell+1)}$ (let's call it $v_i^\dag$), we can find a compact set $\K$ in $D^\dag$ such that (possibly after passing to a subsequence, and translating each element in it in an appropriate sense) the parts of $v_i^\dag$ landing in $\K$ converge in the sense of \cite{fish-target-local}, and we will then derive a contradiction by showing that the limiting curve is non-trivial but has zero ``area'' (in an appropriate sense).

  We next explain how to choose the subsequence $v_i^\dag$ and the set $\K$.
  We first define what it means for one of the curves $v_i^{\ell(\ell+1)}$ or $v_i^k$ to be ``well-situated''.

  To begin with, we define 3 special compact subsets $K^{\ell(\ell+1)}_\topp,K^{\ell(\ell+1)}_\bott\subset K^{\ell(\ell+1)}$ of $D^{\ell(\ell+1)}\subset\hat W^{\ell(\ell+1)}$.
  We take $K^{\ell(\ell+1)}_\topp$ to be a ``closed interval'' inside the neighbourhood of infinity $\Op_{\hat W^{\ell(\ell+1)}}(+\infty)$ (i.e., all points of the form $\set{a\le r_\ell\le b}$ for some $a<b$, where $r_\ell$ is the symplectization parameter on $\s V^\ell\hto \hat W^{\ell(\ell+1)}$).
  Similarly, we take $K^{\ell(\ell+1)}_\bott$ to be a closed interval in $\Op_{\hat W^{\ell(\ell+1)}}(-\infty)$.
  Finally, we take $K^{\ell(\ell+1)}$ to be the union of $K^{\ell(\ell+1)}_\topp$ and $K^{\ell(\ell+1)}_\bott$ and everything ``in between''.
  We define $\partial_\pm K^{\ell(\ell+1)}_\topp$ and $\partial_\pm K^{\ell(\ell+1)}_\bott$ to be the ends of the corresponding intervals.

  \begin{figure}
    \centering
    \begin{minipage}[b]{0.48\linewidth}
      \begin{center}
        \begin{tikzpicture}[scale=1.5,every node/.style={scale=0.5,transform shape},y=-1cm,
          brace/.style={decorate,decoration={brace,amplitude=10pt,mirror},xshift=-1.5pt,yshift=0pt},
          rbrace/.style={decorate,decoration={brace,amplitude=10pt},xshift=2.5pt,yshift=0pt},
          sbrace/.style={decorate,decoration={brace,amplitude=10pt},xshift=15pt,yshift=0pt}]
          \draw (0,0.5) -- (2,0.5);
          \draw (0,0.9) -- (2,0.9);

          \draw (0,1.5) -- (2,1.5);
          \draw (0,1.7) -- (2,1.7);

          \draw (0,2.3) -- (2,2.3);
          \draw (0,2.7) -- (2,2.7);

          \fill[opacity=0.1] (0,0.5) -- (2,0.5) -- (2,0.9) -- (0,0.9);
          \fill[opacity=0.1] (0,2.3) -- (2,2.3) -- (2,2.7) -- (0,2.7);
          \draw (0,0)   -- (0,3);
          \draw (2,0)   -- (2,3);
          \node at (0,0.4) [anchor=west] {\scriptsize${}_{\partial_+K^{\ell(\ell+1)}_\topp}$};
          \node at (1,0.7) {$K^{\ell(\ell+1)}_\topp$};
          \node at (0,1.03) [anchor=west] {\scriptsize${}^{\partial_-K^{\ell(\ell+1)}_\topp}$};
          \node at (0,2.2) [anchor=west] {\scriptsize${}_{\partial_+K^{\ell(\ell+1)}_\bott}$};
          \node at (1,2.5) {$K^{\ell(\ell+1)}_\bott$};
          \node at (0,2.83) [anchor=west] {\scriptsize${}^{\partial_-K^{\ell(\ell+1)}_\bott}$};
          \draw [brace] (0,0) -- (0,1.5) node [midway,xshift=-1.5cm] {$(\hat{W}^{02,\tau})^{\ell}$};
          \draw [brace] (0,1.7) -- (0,3) node [midway,xshift=-1.5cm] {$(\hat{W}^{02,\tau})^{\ell+1}$};
          \draw [rbrace] (2,0.5) -- (2,2.7) node [midway,xshift=1.3cm] {$K^{\ell(\ell+1)}$};
        \end{tikzpicture}
      \end{center}
      \caption{The regions $K^{\ell(\ell+1)}_\topp$ and $K^{\ell(\ell+1)}_\bott$.\\}\label{fig:Kregions}
    \end{minipage}
    \begin{minipage}[b]{0.51\linewidth}
      \begin{tikzpicture}[scale=1.5,every node/.style={scale=0.5,transform shape},y=-1cm,
        brace/.style={decorate,decoration={brace,amplitude=10pt,mirror},xshift=-1.5pt,yshift=0pt},
        rbrace/.style={decorate,decoration={brace,amplitude=10pt},xshift=2.5pt,yshift=0pt},
        sbrace/.style={decorate,decoration={brace,amplitude=10pt},xshift=15pt,yshift=0pt}]
        \draw (0,1.0) -- (2,1.0);
        \draw (0,1.4) -- (2,1.4);

        \draw (0,2.3) -- (2,2.3);
        \draw (0,2.5) -- (2,2.5);

        \fill[opacity=0.1] (0,1.0) -- (2,1.0) -- (2,1.4) -- (0,1.4);
        \filldraw[opacity=0.5, fill=red, draw=black]
        (0,0.2) .. controls (1,0.2) and (1,0.6) .. (2,0.5) --
        (2,0.7) .. controls (1,0.8) and (1,0.5) .. (0,0.4);
        \filldraw[opacity=0.8, fill=green, draw=black]
        (0,0.6) .. controls (1,0.9) and (1,1.4) .. (0,1.9) --
        (0,1.7) .. controls (0.8,1.4) and (0.8,0.9) .. (0,0.8);
        \filldraw[opacity=0.5, fill=red, draw=black]
        (2,1.8) .. controls (1,2.2) and (1,2.5) .. (2,2.9) --
        (2,2.7) .. controls (1.4,2.5) and (1.4,2.2) .. (2,2.0);
        \filldraw[opacity=0.8, fill=green, draw=black]
        (0,2.0) .. controls (1,2.1) and (1,1.4) .. (2,1.5) --
        (2,1.7) .. controls (1,1.7) and (1,2.4) .. (0,2.2);
        \draw (0,0)   -- (0,3);
        \draw (2,0)   -- (2,3);
        \node at (1,1.2) {$K^{\ell(\ell+1)}_\topp$};
        \draw [brace] (0,0) -- (0,2.3) node [midway,xshift=-1.5cm] {$(\hat{W}^{02,\tau})^{\ell}$};
      \end{tikzpicture}
      \caption{Two well-situated and two ill-situated curves $v_i^{\ell(\ell+1)}$.}\label{fig:situated}
    \end{minipage}
  \end{figure}

  We now say that a curve $v_i^{\ell(\ell+1)}$ is \emph{well-situated} if it has a component which either (i) intersects both $\set{t=N}\cap K^{\ell(\ell+1)}$ and $\set{t=N+1}\cap K^{\ell(\ell+1)}$; or (ii) intersects both of $\partial_\pm K^{\ell(\ell+1)}_\topp$; or (iii) intersects both of $\partial_\pm K^{\ell(\ell+1)}_\bott$.

  A curve $v_i^k$ is \emph{well-situated} if it intersects both $\set{t=N}$ and $\set{t=N+1}$ and moreover, in the case $k=1$, it should be contained in a ``vertically bounded region'' -- that is, recalling that $v_i^k$ was the image of a curve in $C^1_{n_i,\tau}$ under the embedding $C^1_{n_i,\tau}\hto D^1$, the curve in $C^1_{n_i,\tau}$ should be contained in a finite interval $a<r_1<b$, where $r_1$ is the symplectization parameter on $\s V^1$.

  We now claim that one of the sequences $v_i^{\ell(\ell+1)}$ or $v_i^k$ has a well-situated subsequence, i.e., a subsequence consisting of well-situated curves; in fact, for each fixed $i$, one of the $v_i^{\ell(\ell+1)}$ or $v_i^k$ is well-situated.

  Indeed, fix $i$ and consider a component of $v_i$ (restricted to $C_{n_i,\tau_i}$) intersecting both $\set{t=N}$ and $\set{t=N+1}$.
  If this component is contained entirely in some $C_{n_i,\tau_i}^k$, then $v_i^k$ is well-situated.
  Otherwise, it must be contained in $K^{\ell(\ell+1)}$ or intersect both of $\partial_\pm K_\topp^{\ell(\ell+1)}$ or both of $\partial_\pm K_\bott^{\ell(\ell+1)}$ (for one of $\ell=0,1$), and hence $v_i^{\ell(\ell+1)}$ is well-situated.

  We now take $v_i^\dag$ to be whichever of the $v_i^{\ell(\ell+1)}$ or $v_i^k$ has a well-situated subsequence.
  Let us also re-index and simply write $v_i^\dag$ for the subsequence.

  Next, we explain our choice of $\K$.
  In the case $\dag=\ell(\ell+1)$, we take $\K$ to be a small compact neighbourhood of $K^{\ell(\ell+1)}$ with smooth boundary.\footnote{%
    The only reason we don't simply take $\K=K^{\ell(\ell+1)}$ is that, for the hypothesis of \cite[Theorem~3.1]{fish-target-local}, $\K$ needs to be a ``compact region'', as in \emph{op. cit.} Definition~2.1.
  }

  In the case $\dag=k$, we take $\K$ to be (a small compact neighbourhood of) $I\times\Sigma\times[N,N+1]\subset\R\times\Sigma\times(N-\varepsilon,N+1+\varepsilon)\cong D^k$ for some interval $I$ (where the displayed isomorphism is given by a symplectization coordinate on $D^k$).

  We now want to show that, in either case, we have the data for and fulfill the assumptions of \cite[Theorem~3.1]{fish-target-local}.

  Referring to the notation from \textit{loc. cit.}, we take $\K$ as chosen above, we take $M$ to be a small open neighbourhood of $\K$ in $D^\dag$, we take $(J,g)$ to be $(\check{J}^\dag,\check{g}^\dag)$, where $\check{g}^\dag$ is the metric induced by $\check\omega^\dag$ and $\check{J}^\dag$ and similarly, $(J_i,g_i)=(\check{J_i}^\dag,\check{g_i}^\dag)$. Let us also set $\omega=\check\omega^\dag$ and $\omega_i=\check\omega^\dag_i$.

  Finally, we explain our choice of $u_i$.
  In the case $\dag=$``$\ell(\ell+1)$'', we simply take $u_i$ to be the restriction of $v_i^\dag$ to $M$.

  In case $\dag=0$, consider a connected component of $v_i^0$ that intersects $\set{t=N}$ and $\set{t=N+1}$.
  We will take $u_i$ to be a (restriction of) a translation of this component with respect to the symplectization parameter (the result is still $J_i$-holomorphic since $J_i$ is translation-invariant).
  Specifically, we translate the component so that its supremum in the symplectization direction is $\frac{I_-+I_+}2$ (and then restrict the result to $M$), where $I_\pm$ are the endpoints of the interval $I$ appearing in the definition of $\K$.

  The case $\dag=2$ is the same, with ``supremum'' replaced by ``infimum''.
  In the case $\dag=1$, we may use either the supremum or infimum, because of our assumption on well-situated curves in this case that they are contained in a ``vertically bounded region''.

  We verify the assumptions of the \cite[Theorem~3.1]{fish-target-local}. By definition $J_i$ converges to $J$ (and hence $g_i$ to $g$). The ``generally immersed'' condition is automatic for us since the domains of the curves $v_i$, when restricted to the region $\set{t\ge T'}$ are compact, and the ``robustly $\K$-proper'' condition holds for the same reason. The condition (2) holds trivially since each $v_i$ has genus 0 domain. So it remains to verify (1).

  We must show that the set $\set{\int u_i^*\omega_i}_{i\in\Z_{>0}}$ is bounded above.

  In the case $\dag=$``$\ell(\ell+1)$'', $u_i$ is (up to certain identifications) a restriction of the curve $v_i$ mapping to $\hat W^{02,\tau_i}$, but we have that $\int v_i^*\tilde\omega^{\tau_i}\toxlim{i}0$ by assumption.

  Now consider the case $\dag=k$, in which $v_i^{k}$ was possibly translated to obtain $u_i$.
  In the region $C^k_{n_i,\tau_i}$, $\tilde\omega^{\tau_i}$ has the form
  \[
    \tilde\omega^{\tau_i}=\dv(e^R\alpha^k)=e^R(\dv\alpha^k+R'\dr\wedge\alpha^k).
  \]
  We still know that $\int(v_i^{k})^*\tilde\omega^{\tau_i}\toxlim{i}0$, and hence, by the positivity of $e^R$, $R'$, $(v_i^{k})^*\dv\alpha^k$, and $(v_i^{k})^*\dr\wedge\alpha^k$, and the boundedness of $R$, we conclude
  \[
    \int u_i^*\dv\alpha^k\le
    \int(v_i^{k})^*\dv\alpha^k
    \toxlim{i}0
  \]
  (where the first inequality comes from the translation-invariance of $\dv\alpha^k$ and the fact that $u_i$ is a translation of a restriction of $v_i^{k}$). Moreover we can choose $R$ (in Lemmas~\ref{lem:interpolating-form}~and~\ref{lem:interpolating-form-glued}) so that $R'$ is bounded as well, and so it remains to bound $\int u_i^*(\dr\wedge\alpha^k)=\int_{\tilde S_i}(v_i^k)^*(\dr\wedge\alpha^k)$, where $\tilde S_i=u_i\I(M)$ (and where the equality come from translation invariance).

  Note that $\rstr{v_i^k}{\tilde S_i}$ maps into a subset $\set{r_0\le r\le r_1}\subset\s V^k$ where $r_1-r_0$ is bounded independently of $i$ (namely, it is bounded be the length of $I$). We will in fact bound $A_i=\int_{(v_i^k)\I\set{r_0\le r\le r_1}}(v_i^k)^*\dr\wedge\alpha^k$.

  Let us write $f$ for the composite of $v_i^k$ with the symplectization parameter $\s V\to\R$. Note that, if $f$ had no critical values, we would have
  \[
    A_i=\int_{r_0}^{r_1}\plr{\int_{f\I(r)}(v_i^k)^*\alpha^k}\dr
  \]
  In fact, this is still true for any $f$, but where now the integrand is defined outside the measure zero set of critical values $r$ (see, e.g., \cite[p.210]{sulanke-wintgen-diff-geo-und-faser}). Hence, it remains to bound $\int_{f\I(r)}(v_i^k)^*\alpha^k$. But by Stokes' theorem, $\int_{f\I(r)}(v_i^k)^*e^{R(r)}\alpha^k<Ca$ (where $Ca$ is the bound on $\int v_i^*\tilde\omega^{\tau_i}$ from the beginning of the proof). Since $e^R$ is bounded below, we obtain the desired bound.

  We have now verified the hypotheses of \cite[Theorem~3.1]{fish-target-local}. Hence, the theorem guarantees us a subsequence of $u_k$ which robustly $\K$-converges to some holomorphic curve $u$. We then obtain a contradiction as follows.

  In the case $\dag=\ell(\ell+1)$, we have that $\int u^*\omega=0$ but that $u$ is non-constant with non-empty domain by Lemma~2.2.3 below, which is impossible since $u$ is $J$-holomorphic and $J$ is $\omega$-tame.\footnote{In this case, we could also derive the contradiction from Gromov's monotonicity lemma. The more complicated argument using \cite[Theorem~3.1]{fish-target-local} is needed in the case $\dag=k$ since in this case the curves $v_i^k$ cannot be assumed to (have components that) remain in a compact region.}
  Here, to apply Lemma~2.2.3~\ref{item:fish-contradiction-nonempty}, we take $K_1$ and $K_2$ to be either (i) $K_1=\K\cap\set{t=N}$ and $K_2=\K\cap\set{t=N+1}$, or (ii) $\partial_\pm K^{\ell(\ell+1)}_\topp$, or (iii) $\partial_\pm K^{\ell(\ell+1)}_\bott$.
  That we have a subsequence of curves each of which has a component intersecting both $K_1$ and $K_2$ follows from our definition of ``well-situated curve''.

  In the case $\dag=k$, we instead use that $\int u^*\dv\alpha_\infty=0$, since $\int u_i^*\dv\alpha^k\tox{i\to\infty}0$, as noted above, and again that $u$ is non-constant with non-empty domain.
  This time, we are taking $K_1$ and $K_2$ in Lemma~2.2.3~\ref{item:fish-contradiction-nonempty} either as in (i) above, or we take $K_i=\K\cap\set{r=r_i}$ for $i=1,2$, where again by the definition of ``well-situated'', such $r_1,r_2$ can be chosen so that a subsequence of the $u_i$ intersect both $K_1,K_2$.

  It follows that $u$ must be contained in a product of $\R$ with a union of Reeb flow lines in $\Sigma\times(N-\varepsilon,N+1+\varepsilon)$.
  But this is impossible since, by assumption, each $u_k$, and hence $u$, has supremum or infimum $\frac{I_-+I_+}2$ in the $\R$ direction, but the boundary of $u$ must lie outside of $\K$ by Lemma~\ref{lem:fish-contradiction}~\ref{item:fish-contradiction-compact}.
\end{proof}

\begin{lem}\label{lem:fish-contradiction}
  Let $M,J,g,J_k,g_k,\K,u_k$ be as in \cite[Theorem~3.1]{fish-target-local}, and let $u$ be the robust $\K$-limit of the sequence $u_k$.
  \begin{enumerate}[(i)]
  \item\label{item:fish-contradiction-compact} The domain of $u$ is a compact nodal Riemann surface with boundary, and $u\I(\K)$ is disjoint from the boundary.
  \item\label{item:fish-contradiction-nonempty} If there exists disjoint compact subsets $K_1,K_2\subset\K$ such that a connected component of each $\rstr{u_i}{u_i\I(\K)}$ intersects both $K_1$ and $K_2$, then $\int u^*\omega\ne0$.
  \item\label{item:fish-contradiction-integral} Let $\beta_i\in\Omega^2(M)$ be a sequence of 2-forms with $u_i^*\beta_i\ge0$ for all $i$ and $\int u_i^*\beta_i\tox{i\to\infty}0$.
    Then, if the $\beta_i$ converge in $\cc^\infty_\loc$ to some $\beta\in\Omega^2(M)$, then $\int u^*\beta=0$.
  \end{enumerate}
\end{lem}
\begin{proof}
  \ref{item:fish-contradiction-compact} is simply part of the definition of robust $\K$-convergence.

  \ref{item:fish-contradiction-nonempty} will follow if we can show that some connected component of $u$ still intersects both $K_1$ and $K_2$, since then $u$ is non-constant on that component.

  The definition of robust $\K$-convergence says that certain restrictions of the $u_k$ to a certain subset $\wt{\K}\subset M$ containing $\K$ converge in $\cc^\infty_\loc$, after reparametrization.

  That the limit $u$ still intersects both $K_1$ and $K_2$ follows immediately from the $\cc^0$-convergence (where the restriction to $\wt\K$ does no harm, since the restricted maps will still have a connected component intersecting both $K_1$ and $K_2$).
  Similarly, \ref{item:fish-contradiction-integral} follows immediately from $\cc^1$-convergence.
\end{proof}

\subsubsection{Three more compactness results}
Having established the previous proposition, we now consider three simpler versions. Together, the four of these correspond to Setups~I-IV from \cite[\S1.6]{pardon-contact-homology-and-vfc} (see also \S\S\ref{subsubsec:ch-conv-setups}), but in reverse order.

\begin{propn}
  Let $(\hat W,\set{\omega^\tau}_{\tau\in[0,1]})$ be a family of symplectic cobordisms from $(V^+,\xi^+,X^+,t^+)$ to $(V^-,\xi^-,X^-,t^-)$, and fix admissible pairs $(\alpha^\pm,J^\pm)$ on $V^\pm$ and a family $\set{\hat J^{\tau}}_{\tau\in[0,1]}$ of admissible almost complex structures on $\hat W$.

  Also fix a constant $a>0$.

  Then there exists $T>0$ such that, for any $\tau\in[0,1]$, any relevant $\hat J^\tau$-holomorphic curve in $\hat W$ with positive end a Reeb orbit of action $<a$ is contained within $\set{t<T}\subset\hat W$.
\end{propn}
\begin{proof}
  The proof is the same as for the previous proposition, but simpler, as one only needs to consider three regions $(\hat W)^0$, $(\hat W)^{01}$, and $(\hat W)^1$ instead of five as in Figure~\ref{fig:regions}.

  Also, in this case, we use Lemma~\ref{lem:interpolating-form} in place of Lemma~\ref{lem:interpolating-form-glued}.
\end{proof}

The next two propositions are just special cases of the previous one.

\begin{propn}\label{propn:compactness-cob}
  Let $(\hat W,\hat\omega)$ be a symplectic cobordisms from $(V^+,\xi^+,X^+,t^+)$ to $(V^-,\xi^-,X^-,t^-)$, and fix admissible pairs $(\alpha^\pm,J^\pm)$ on $V^\pm$ and an admissible almost complex structure $\hat J$ on $\hat W$.

  Also fix some $a>0$.

  Then there exists $T>0$ such that any relevant $\hat J$-holomorphic curve in $\hat W$ with positive end of action $<a$ is contained within $\set{t<T}\subset\hat W$.\qed
\end{propn}

\begin{propn}\label{propn:compactness-sympl}
  Let $(V,\xi,X,t,\alpha,J)$ be a contact manifold with a convex structure, admissible representing one-form, and admissible almost complex structure. Fix a constant $a>0$.

  There exists $T>0$ such that any relevant $\hat J$-holomorphic curve in $\s V$ with positive end a Reeb orbit of action $<a$ is contained within $\set{t<T}\subset\s V$.\qed
\end{propn}

\subsubsection{Holomorphic buildings}
As the moduli spaces in \cite{pardon-contact-homology-and-vfc} consist of holomorphic \emph{buildings} (see \textit{op. cit.} and \cite{compactness-in-sft} for the definition) and not just holomorphic curves, we need to ensure the above results hold for the former as well. However, this follows easily from what we already have:

\begin{cor}\label{cor:fish-argument-buildings}
  Propositions~\ref{propn:compactness-glued}-\ref{propn:compactness-sympl} remain true with ``holomorphic curve'' replaced by ``holomorphic building''.

  (We will refer below to the four claims in this corollary (corresponding to Propositions~\ref{propn:compactness-glued}-\ref{propn:compactness-sympl}, in reverse order!) as Corollary~\ref{cor:fish-argument-buildings}~(I)-(IV).)
\end{cor}
\begin{proof}
  The proof is the same in all four cases. By Stokes' theorem, each component of a relevant holomorphic building with positive end of action $<a$ will be a relevant holomorphic curve (possibly in $\s V^\pm$, $\s V^k$, or $\hat W^{\ell(\ell+1)}$) with a positive end of action $<a$. Hence, applying the appropriate instances of Propositions~\ref{propn:compactness-glued}-\ref{propn:compactness-sympl} and taking the maximum of all the resulting $T$, we constrain all of the component holomorphic curves, and hence the holomorphic building, to $\set{t<T}$.
\end{proof}

\subsubsection{Relative versions}\label{subsubsec:rel-compactness}
Finally, we will need the following ``relative'' version of the last two compactness results (i.e., versions (I) and (II)), in which we only fix the complex structure on a large compact subset.

\begin{propn}\label{propn:compactness-sympl-rel}
  Let $(V,\xi,X,t,\alpha,J)$ be a contact manifold with a convex structure, admissible representing one-form, and admissible almost complex structure. Fix a constant $a>0$.

  There exists $T>0$ such that, for any admissible almost complex structure $K$ on $\xi$ which agrees with $J$ on $\set{t<T}\subset V$, any relevant $\hat K$-holomorphic building in $\s V$ with positive end of action $<a$ is contained within $\set{t<T}\subset\s V$.
\end{propn}
\begin{proof}
  The statement is deduced from the corresponding statement with holomorphic curves instead of buildings in the same way that Corollary~\ref{cor:fish-argument-buildings}~(I) is deduced from Proposition~\ref{propn:compactness-sympl}.

  The proof of the statement with holomorphic curves is proven in the same way as Proposition~\ref{propn:compactness-glued}. The only difference is, in the proof by contradiction, in addition to considering a sequence $T_i\tox{i\to\infty}\infty$ one must also consider a sequence of complex structures $K_i$ agreeing with $J$ on $\set{t<T_i}$, and then consider a sequence $v_i$ of $K_i$-holomorphic curves, rather than $J$-holomorphic curves.

  However, the rest of the proof is exactly the same; in particular, all further mention of $J$ refers to its restriction to $\set{t<T}$, where it is equal to $K_i$, and so the same arguments go through with $K_i$.

  The one exception is the application of Lemma~\ref{lem:interpolating-form} to produce a $\hat J$-compatible symplectic form $\tilde\omega$; here, it must be $\hat K_i$-compatible for all $i$. However, inspecting the proof of that lemma, we see that the only way in which the construction of $\tilde\omega$ depends on the given complex structure $\hat J$ is through the sets $\Op(\pm\infty)$ on which it is ``cylindrical'' (i.e., equal to $\wh{J^\pm}$). In the present case, all of the $\hat K_i$ are cylindrical on all of $\s V$, and so the lemma produces a single $\tilde\omega$ which is $\hat K_i$-compatible for all $i$.
\end{proof}

\begin{propn}\label{propn:compactness-cob-rel}
  Let $(\hat W,\hat\omega)$ be a symplectic cobordism from $(V^+,\xi^+,X^+,t^+)$ to $(V^-,\xi^-,X^-,t^-)$, and fix admissible pairs $(\alpha^\pm,J^\pm)$ on $V^\pm$ and an admissible almost complex structure $\hat J$ on $\hat W$.

  Fix constants $M_\pm\in\R$ and $a>0$.

  Then there exists $T>0$ such that all Reeb orbits of $\alpha^+$ of action $<a$ are contained in $\set{t^+<T}\subset V$, and having the following property.

  Fix any representing contact forms $\tilde\alpha^\pm$ and $\dv\tilde\alpha^\pm$-compatible almost complex structures $\tilde J^\pm$ on $V^\pm$ which agree with $\alpha^\pm$ and $J^\pm$ on $\set{t^\pm<T}\subset V^\pm$, and let $r_\pm\colon\s V^\pm\to\R$ (and $r_\pm\colon j_\pm(\s V^\pm)\to\R$) be the corresponding symplectization coordinates.

  Next, fix an almost complex structure $\hat K$ on $\hat W$ which agrees with $\hat J$ on $\set{t<T}\subset\hat W$ and which is equal to $(j_\pm)_*\wh{\tilde J^\pm}$ on $\set{\pm\tilde r_\pm\ge M_\pm}$ (where $\tilde r_\pm$ are the symplectization coordinates coming from $\tilde\alpha^\pm$).

  Then any relevant $\hat K$-holomorphic building in $\s V$ with positive end a Reeb orbit of action $<a$ is contained within $\set{t<T}\subset\hat W$.
\end{propn}
\begin{proof}
  The proof is essentially the same as that of Proposition~\ref{propn:compactness-sympl-rel} (which is a special case).

  Here, in addition to the sequence $T_i\tox{i\to\infty}\infty$, we must consider a sequence $(\tilde\alpha_i^\pm,\tilde J_i^\pm)$ of pairs on $V^\pm$, and a sequence of almost complex structures $\hat K_i$ on $\hat W$, all satisfying the given conditions.

  Again, this does not change the rest of the proof (of Proposition~\ref{propn:compactness-cob}), since the assumption is that the $(\tilde\alpha_i^\pm,\tilde J_i^\pm)$ and the $\hat K_i$ agree with $(\alpha,J)$ and with $\hat J$ on the region $\set{t<T_i}$ on which these data are used.

  The one exception again is that the argument requires a single symplectic form $\tilde{\omega}$, whereas applying Lemma~\ref{lem:interpolating-form} potentially gives us a different $\tilde\omega_i$ for each $\hat K_i$. This time, we cannot necessarily guarantee that all of the $\tilde\omega_i$ are equal, but we \emph{can} arrange that they all be equal to some fixed $\tilde\omega$ on $\set{t<T_i}$, and this is good enough.

  Indeed, as in the proof of Proposition~\ref{propn:compactness-sympl-rel}, we recall that the only dependency of the form $\tilde\omega$ constructed in Lemma~\ref{lem:interpolating-form} on the given almost complex structure $\hat J$ is through the constants $N_\pm$ such that $\hat J$ is equal to $\wh{J^\pm}$ on $\pm r_\pm\ge N_\pm$.

  Here, we are explicitly assuming that all of the $\hat K_i$ are equal to $\wh{\tilde J_i^\pm}$ on $\pm \tilde r_\pm\ge M_\pm$. This implies that we can choose the same functions $R_\pm$ in Lemma~\ref{lem:interpolating-form} for all the $\hat K_i$. But then, on $\set{t<T_i}$, where $\tilde\alpha^\pm=\alpha^\pm$, the resulting forms $\tilde\omega_i$ will all just be equal to $\dv(e^{R_\pm(r_\pm)}\alpha^\pm)$ for $\pm r_\pm\ge M_\pm$, and equal to $\hat\omega$ elsewhere on $\set{t<T_i}$.
\end{proof}

\subsection{Interpolation}\label{subsec:inter}
In this section, we prove a technical result allowing us to interpolate, in a certain sense, between any two convex structures.

Let $(X,t)$ be a convex structure on $(V,\xi)$.

\begin{lem}\label{lem:scaling-form}
  Let $\alpha$ be an $(X,t)$-admissible form. Then for any smooth non-decreasing function $g\colon\R\to[1,C]$ such that $g(t)=1$ for $t\le T$ (with $T$ as in Definition~\ref{defn:admissible}), the contact forms $\alpha$ and $g(t)\cdot\alpha$ have the same Reeb orbits.
\end{lem}
\begin{proof}
  By condition~\ref{item:admissible-scale} of Definition~\ref{defn:admissible}, we know that for any such $g$, $g(t)\cdot\alpha$ satisfies conditions~\ref{item:admissible-s}~\ref{item:admissible-t}, and thus the result follows from Proposition~\ref{propn:admiss-reeb-orbits}.
\end{proof}

Now let $(X^\pm,t^\pm)$ be two convex structures on $(V,\xi)$.
\begin{propn}\label{propn:interpolating-form}
  Let $\alpha^\pm$ be representing forms which are $(X^\pm,t^\pm)$-admissible. Then for any $a,T^+>0$, there exists a form $\alpha$ representing $\xi$ with the following properties:
  \begin{enumerate}[(i)]
  \item\label{item:interpolating-form-alpha-plus} $\alpha=\alpha^+$ on $\set{t^+<T^+}$.
  \item\label{item:interpolating-form-alpha-minus} There exist $T^->0$ and $C\ge1$ such that $\alpha$ is equal to $C\alpha^-$ on $\set{t^->T^-}$.
  \item\label{item:interpolating-form-action} The Reeb orbits of $\alpha$ and $\alpha^+$ of action $<a$ are the same.
  \end{enumerate}
\end{propn}
\begin{proof}
  We may assume without loss of generality that $T^+$ is greater than the $T$ with respect to which $\alpha^+$ is $(X^+,t^+)$-admissible (see Definition~\ref{defn:admissible}).

  Recall that every contact form representing $\xi$ is of the form $G\alpha^+$ for a positive function $G$; so we must find an appropriate $G$.

  Of course, we must take $G=1$ on $\set{t^+<T^+}$ and, no matter our choice of $T^-$ and $C$, we will need to take $G=C\frac{\alpha^-}{\alpha^+}$ on $\set{t^->T^-}$. Moreover, we will want to take $G=C$ on $\set{T_0<t<T_1}$ for some $T^+<T_0<T_1<T^-$. (The reader is encouraged to draw a picture.)

  Now fix $T_1>T_0>T^+$, and fix $T^->T_1$ large enough so that $\set{t^+<T^+}\cap\set{t^->T^-}=\emptyset$, and fix a smooth function $H$ on $\set{t\ge T_1}$ which is equal to $1$ on a neighbourhood of $\set{t^+=T_1}$ and is equal to $\frac{\alpha^-}{\alpha^+}$ on $\set{t^->T^-}$.

  Next, for any $C\ge1$, define $G_C$ on $V$ to be equal to $1$ on $\set{t^+\le T^+}$, equal to $C$ on $\set{T_0\le t^+\le T_1}$, equal to $C\cdot H$ on $\set{t\ge T_1}$, and equal to an arbitrary smooth monotone function of $t$ on $\set{T^+\le t^+\le T^0}$.

  Let us now see that, for $C$ large enough, $G=G_C$ is as required. Conditions \ref{item:interpolating-form-alpha-plus} and \ref{item:interpolating-form-alpha-minus} are immediate, so let us consider \ref{item:interpolating-form-action}.

  By Lemma~\ref{lem:scaling-form}, we know that $G_1\alpha^+$ has the same Reeb orbits as $\alpha^+$ on $\set{t^+\le T_1}$. Hence, any Reeb orbit $\gamma$ of $G_1\alpha^+$ which is not a Reeb orbit of $\alpha^+$ will have to intersect $t^+>T_1$. Let $a_0$ be the infimal action of the restriction of any such orbit to $\set{t^+\ge T_0}$. We claim that $a_0>0$. Indeed, if $\gamma$ is contained entirely inside $\set{t^+\ge T_0}$, this follows from the fact that (by Proposition~\ref{propn:admiss-reeb-orbits}) all of the Reeb of $\alpha^-$ (and hence $G_\alpha$) are contained in a compact set, and otherwise $\gamma$ must traverse $\set{T_0\le t^+\le T_1}$, which requires some minimal action.

  Similarly, we find that the minimal action of any Reeb orbit of $G_C\alpha$ which is not a Reeb orbit of $\alpha$ must be $\ge Ca_0$, so choosing $C>a/a_0$, we are done.
\end{proof}

\section{Defining contact homology}\label{sec:defining-ch}
Having introduced admissible forms and complex structures, and having obtained the needed compactness results, we are now ready to introduce contact homology for convex open contact manifolds, following the procedure in \cite{pardon-contact-homology-and-vfc}.

To do this, we proceed in two steps: first, we show that we can define an invariant of a fixed convex structure, and then we show that these in fact agree for different convex structures.

To explain why this is needed, recall that invariance with respect to different contact forms is obtained via functoriality with respect to symplectic cobordisms. However, on cobordisms between complex structures adapted to different convex structures, we have no way to control the holomorphic curves.

However, when the convex structure is the same at both ends, then we have the required Corollary~\ref{cor:fish-argument-buildings}, which is why we are able to obtain an invariant of a fixed convex structure.

To mediate between different convex structures, we will make use of the interpolating forms from \S\ref{subsec:inter}. Specifically, given $\alpha^+$ and $\alpha^-$ adapted to different convex structures, we will take $\alpha$ which (i) agrees with $\alpha^-$ at infinity, and (ii) agrees with $\alpha^+$ on a large compact set. By (i), we are allowed to compare the contact homology associated to $\alpha$ and $\alpha^-$, as they are admissible to the same convex structure. But now by (ii), using the ``relative'' compactness results of \S\ref{subsubsec:rel-compactness}, we can insure that the contact homology associated to $\alpha$ and $\alpha^+$ agree, \emph{if we restrict to Reeb orbits up to a fixed action}.

Thus, we obtain a morphism from an action-bounded contact homology associated to $\alpha^+$ to the contact homology associated to $\alpha^-$. By a colimiting procedure, we can then get rid of the action bound. It then remains to see that the morphisms thus obtained are canonical -- i.e., independent of the interpolating forms $\alpha$ -- and that they satisfy the required commutativity property needed to obtain a canonical invariant.

\subsection{$\CH$ for a fixed convex structure}\label{subsec:ch-for-fixed-convex}
We will now define a $\CH$ for a contact manifold with a fixed convex structure as the contact homology of any form (and complex structure) admissible with respect to it, and show that it is functorial with respect to cobordisms between manifolds with matching convex structures (and in particular, that it is independent of the chosen form).

This will entail precisely repeating the construction from \cite[\S1.7]{pardon-contact-homology-and-vfc}, and indicating where the necessary changes need to be made.

Then, in \S\ref{subsec:ch-for-open}, we will show that it is in fact independent up to isomorphism of the chosen convex structure, hence giving an invariant of the contact structure alone.

\subsubsection{The Setups}\label{subsubsec:ch-conv-setups}
To begin with, let us describe the relevant versions of Setup~I-IV as in \cite[\S1.6]{pardon-contact-homology-and-vfc}:
\\

\begin{setup}\label{setup:1}
consists of a contact manifold $V$ equipped with a convex structure $(X,t)$ and an admissible pair $(\alpha,J)$.
\end{setup}

\begin{setup}\label{setup:2}
consists of an exact symplectic cobordism $(\hat W,\dv\hat\lambda)$ from $(V^+,\xi^+,X^+,t^+)$ to $(V^-,\xi^-,X^-,t^-)$ relative to some $i\colon\Op_{V^+}(\infty)\to\Op_{V^-}(\infty)$, together with admissible pairs $(\alpha^\pm,J^\pm)$ on $V^\pm$, and an admissible almost complex structure $\hat J$ on $\hat W$.
\end{setup}

\begin{setup}\label{setup:3}
consists of a one-parameter family of exact symplectic cobordisms $(\hat W,\set{\dv\hat\lambda^\tau}_{\tau\in[0,1]})$ with positive/negative ends $(V^\pm,\xi^\pm,X^\pm,t^\pm)$ (relative to some fixed $i$), together with admissible pairs $(\alpha^\pm,J^\pm)$ on $V^\pm$, and an admissible family of almost complex structures $\hat J^\tau$ on $\hat W$.
\end{setup}

\begin{setup}\label{setup:4}
  consists of a one-parameter family of exact symplectic cobordisms $(\hat W^{02,\tau},\dv\hat\lambda^{02,\tau})_{\tau\in{}[0,\infty)}$ from $(V^0,\xi^0,X^0,t^0)$ to $(V^2,\xi^2,X^2,t^2)$ (relative to some fixed $i$), which for sufficiently large $\tau$ agrees with the $\tau$-gluing of a symplectic cobordism $(\hat W^{01},\dv\hat\lambda^{01})$ from $V^0$ to some $(V^1,\xi^1,X^1,t^1)$ and a symplectic cobordism $(\hat W^{12},\dv\hat\lambda^{12})$ from $V^1$ to $V^2$, together with admissible pairs $(\alpha^k,J^k)$ on $V^k$ ($k=0,1,2$) and admissible almost complex structures $\hat J^{\ell(\ell+1)}$ ($\ell=0,1$), and an admissible family of almost complex structures $\hat J^{02,\tau}$ (agreeing with the one induced from $\hat J^{01}$ and $\hat J^{02}$ for large $\tau$).
\end{setup}
Note that each instance of our Setups is an instance of the corresponding Setup in \cite{pardon-contact-homology-and-vfc} (except that our contact manifolds
are not compact) and hence, we obtain the moduli spaces $\Mbar_{\rI}(T)$, $\Mbar_{\rII}(T)$, $\Mbar_{\rIII}(T)$, and $\Mbar_{\rIV}(T)$ as in \cite[\S2.3]{pardon-contact-homology-and-vfc}. Moreover, by Corollary~\ref{cor:fish-argument-buildings},
each of these moduli spaces is compact.

As this was the only use of the compactness of the contact manifolds considered in \cite{pardon-contact-homology-and-vfc}, the main Theorem~1.1 of \textit{op. cit.} goes through with the present definition of the four setups, giving us the sets $\Theta$ corresponding to any instance of one of these setups, and virtual moduli counts $\Mbar^\vir_\theta$ for $\theta\in\Theta$.

\subsubsection{Defining the invariant}\label{subsubsec:defining-CH-conv}
With this in place, we now consider the constructions (1.23)-(1.27) from \cite[\S1.7]{pardon-contact-homology-and-vfc}. These go through unchanged to give us the following objects, given the data from Setups~\ref{setup:1}-\ref{setup:4}.

Given data as in Setup~\ref{setup:1}, we obtain a supercommutative unital $\Q$-superalgebra
\[
  \CH_\bullet(V,\xi;X,t)_{\alpha,J,\theta}
\]
for each $\theta \in \Theta_{\rI}(V,\alpha,J)$, which we may also denote by $\CH_\bullet(V,\xi)_{\alpha,J,\theta}$. (We also obtain a dg-superalgebra from which the above is obtained by passing to homology, and which we denote using $\CC_\bullet$ instead of $\CH_\bullet$.)

Given data as in Setup~\ref{setup:2}, we obtain a graded $\Q$-algebra map
\begin{equation}\label{eq:cobordism-map}
\CH_\bullet(V^+,\xi^+;X^+,t^+)_{\alpha^+,J^+,\theta^+}\tox{\Phi(\hat W,\dv\hat\lambda)_{\hat J,\theta}}\CH_\bullet(V^-,\xi^-;X^-,t^-)_{\alpha^-,J^-,\theta^-}
\end{equation}
for any $\theta \in \Theta_{\rII}(\hat W,\dv\hat\lambda,\hat J)$ mapping to $\theta^\pm \in \Theta_{\rI}(V^\pm, \alpha^\pm, J^\pm)$.

Whenever we have data as in Setup~\ref{setup:3}, we can conclude that the following two maps coincide:
\begin{equation}\label{eq:ch-morphism-homotopy}
  \begin{tikzcd}[column sep = large]
    \CH_\bullet(V^+,\xi^+;X^+,t^+)_{\alpha^+,J^+,\theta^+}\ar[shift left=0.6ex]{r}{\Phi(\hat W,\dv\hat\lambda^0)_{\hat J^0,\theta^0}}\ar[shift left=-0.6ex]{r}[swap]{\Phi(\hat W,\dv\hat\lambda^1)_{\hat J^1,\theta^1}}&\CH_\bullet(V^-,\xi^-;X^-,t^-)_{\alpha^-,J^-,\theta^-}.
  \end{tikzcd}
\end{equation}
By Lemma~\ref{lem:acs-family-cobordisms}, any pair of admissible almost complex structures $(\hat J^0, \hat J^1)$ on a given exact symplectic cobordism $(\hat W, \dv\hat\lambda)$ can be extended to an admissible family $\{\hat J^\tau\}_{\tau \in [0,1]}$. It follows that the required data for Setup~\ref{setup:3} can always be obtained, and hence that the map \eqref{eq:cobordism-map} is independent of $\hat J$ and $\theta$ and can therefore be written as
\[
  \CH_\bullet(V^+,\xi^+;X^+,t^+)_{\alpha^+,J^+,\theta^+}\tox{\Phi(\hat W,\dv\hat\lambda)}\CH_\bullet(V^-,\xi^-;X^-,t^-)_{\alpha^-,J^-,\theta^-}.
\]

Finally, given data as in Setup~\ref{setup:4}, we obtain a commutative diagram
\begin{equation}\label{eq:ch-triangle}
  \begin{tikzcd}[row sep=30pt, column sep=-10pt]
    &\CH_\bullet(V^1,\xi^1;X^1,t^1)_{\alpha^1,J^1,\theta^1}\ar{rd}{\Phi(\hat W^{12},\dv\hat\lambda^{12})}\\
    \CH_\bullet(V^0,\xi^0;X^0,t^0)_{\alpha^0,J^0,\theta^0}\ar{rr}[swap]{\Phi(\hat W^{02,0},\dv\hat\lambda^{02,0})}\ar{ru}{\Phi(\hat W^{01},\dv\hat\lambda^{01})}&&\CH_\bullet(V^2,\xi^2;X^2,t^2)_{\alpha^2,J^2,\theta^2}.
  \end{tikzcd}
\end{equation}

Now, given any two contact forms $\alpha^+$ and $\alpha^-$ on a contact manifold $(V,\xi)$, we can consider the trivial cobordism $\s V$ from $(V,\alpha^+)$ to $(V,\alpha^-)$. If $\alpha^+$ and $\alpha^-$ are both admissible with respect to the same convex structure $(X,t)$ on $V$, then there exists an admissible almost complex structure on the trivial cobordism by Lemma~\ref{lem:acs-trivial-cobordism}. It follows that \cite[Lemma~1.2]{pardon-contact-homology-and-vfc} still holds in our context (the proof goes through unchanged), and we conclude that the objects $\CH_\bullet(V,\xi;X,t)_{\alpha,J,\theta}$ are canonically isomorphic for different choices of $(\alpha,J,\theta)$; we thus obtain a well-defined object only depending on the contact structure $\xi$ and the convex structure $(X,t)$, which we could thus denote by $\CH_\bullet(V,\xi;X,t)$.
However, we do not want to want to dwell on this since below (in \S\ref{subsec:ch-for-open}), we will show that it is in fact also independent of $(X,t)$.

\begin{rmk}
  Given any deformation class of symplectic cobordisms from $(V^+,\xi^+;X^+,t^+)$ to $(V^-,\xi^-;X^-,t^-)$, we obtain a morphism $\CH_\bullet(V,\xi;X^+,t^+)\to\CH_\bullet(V^-,\xi^-;X^-,t^-)$. Moreover, these are functorial with respect to gluings of cobordisms, and hence we obtain a (symmetric monoidal) functor on a certain symplectic cobordism category, which we will not need or say anything more about.
\end{rmk}

\begin{rmk}\label{rmk:ch-details}
  In what follows, we will need to use some details about the actual construction of $\CH_\bullet$ and the induced morphisms $\Phi$, about which we have said nothing so far -- most importantly, that $\CC_\bullet(V,\xi)_{\alpha,J,\theta}$ is the free supercommutative $\Q$-superalgebra generated by the ``good'' Reeb orbits of $\alpha$, and that the differential is defined in terms of certain (virtual!) counts of $\hat J$-holomorphic buildings in $\s V$. Similarly, the morphism $\Phi(\hat W,d\hat\lambda)_{\hat J,\theta}$ is defined using virtual counts of $\hat J$-holomorphic buildings in $\hat W$. We refer to \cite[\S\S1.2-1.5]{pardon-contact-homology-and-vfc} for the details.
\end{rmk}

\subsubsection{Action bounded version}\label{subsubsec:ch-conv-action-bounded}
Next, we define ``action-bounded'' versions of this invariant. For any fixed $V,\xi,X,t,\alpha$, we can repeat the above construction of $\CH_\bullet(V,\xi)_{\alpha,J,\theta}$ using only Reeb orbits of $\alpha$ with action $<a$ to obtain dg-algebras $\CC_\bullet^{<a}(V,\xi)_{\alpha,J,\theta}$ and homology algebras $\CH_\bullet^{<a}(V,\xi)_{\alpha,J,\theta}$. That this is well-defined is due to the fact that the holomorphic curves defining the differential live in the trivial cobordism $\s V$, on which the form $\dv\alpha$ is defined and positive on holomorphic curves, and hence the differential lowers the action by Stokes' theorem.

The complexes $\CC^{<a}_\bullet(V,\xi)_{\alpha,J,\theta}$ give a filtration of $\CC_\bullet(V,\xi)_{\alpha,J,\theta}$ by subcomplexes, and hence we have (since directed colimits commute with taking homology) that $\CH_\bullet(V,\xi)_{\alpha,J,\theta}$ is the colimit of the $\CH^{<a}_\bullet(V,\xi)_{\alpha,J,\theta}$.

Next, suppose we are given two $(X,t)$-admissible pairs $(\alpha^\pm,J^\pm)$. Then by Remark~\ref{rmk:admiss-rel-bounded}, the function $\alpha^-/\alpha^+$ is bounded. Fix $a>0$ and $b\ge\sup(\frac{\alpha^-}{\alpha^+})a$. We claim that the morphism $\Phi(\s V,\hat\omega)\colon\CH_{\bullet}(V,\xi)_{\alpha^+,J^+,\theta^+}\to\CH_{\bullet}(V,\xi)_{\alpha^-,J^-,\theta^-}$ (for any choice of $\theta^\pm)$ restricts to a morphism
\[
  \CH_\bullet^{<a}(V,\xi)_{\alpha^+,J^+,\theta^+}
  \tox{\Phi(\s V,\hat\omega)}
  \CH_\bullet^{<b}(V,\xi)_{\alpha^-,J^-,\theta^-},
\]
i.e., given a Reeb orbit $\gamma^+$ of action $<a$, seen as a generator of $\CH_\bullet^{<a}(V,\xi)_{\alpha^+,J^+,\theta^+}$, any Reeb orbit appearing in its image under $\Phi(\s V,\hat\omega)$ must have action $<b$.

Indeed, letting $r_\pm$ be the symplectization coordinates on $\s V$ associated to $\alpha^\pm$, if we choose any $M_\pm\in\R$ with $e^{M^++M^-}=\sup(\frac{\alpha^-}{\alpha^+})+\varepsilon$ (for any $\varepsilon>0$), the sets $\set{r_+\ge M_+}$ and $\set{r_-\le-M_-}$ in $\s V$ will be disjoint (recall that $e^{r_+}\alpha^+=e^{r_-}\alpha^-$). We can thus choose the $\hat J$ defining $\Phi(\s V,\hat\omega)=\Phi(\s V,\hat\omega)_{\hat J,\hat\theta}$ (as in Lemma~\ref{lem:acs-trivial-cobordism}) to be equal to $\wh{J^\pm}$ on $\set{\pm r_\pm\ge M_\pm}$. It then follows that we can choose an exact $\hat J$-compatible form $\tilde\omega$ as in Lemma~\ref{lem:interpolating-form} equal to $\dv(e^{\pm M_\pm\pm\varepsilon}\alpha^\pm)$ on $\Op_{\s V}(\pm\infty)$, and hence by applying Stokes' theorem, we have (writing $\Ac$ for the action functional) that $e^{M_+}\Ac(\gamma^+)>e^{-M_-}\Ac(\gamma^-)$ (and hence $\Ac(\gamma^-)<e^{M^++M^-}\Ac(\gamma^+)<\plr{\sup\frac{\alpha^-}{\alpha^+}+\varepsilon}a$) for any Reeb orbits $\gamma^\pm$ of $\alpha^\pm$ such that there is a $\hat J$-holomorphic building in $\s V$ with a single positive puncture limiting to $\gamma^+$, and with one of its negative punctures limiting to $\gamma^-$.\footnote{Here, we are using that the coefficients appearing in the homomorphism $\Phi(\s V,\hat\omega)$ are ``counts'' of holomorphic buildings connecting the given Reeb orbits; although in this case these are \emph{virtual} counts, it is still the case that if a certain coefficient is non-zero, then there must be at least one such holomorphic building; see \cite[Theorem~1.1~(iv)]{pardon-contact-homology-and-vfc}. Similar comments apply to the definition of $\CH_\bullet^{<a}$ above.}

Similarly, given two such choices of $\hat J$, we can choose an interpolating family $(\hat J^\tau)_{\tau\in[0,1]}$ (as in Lemma~\ref{lem:acs-family-cobordisms}) so that each $\hat J^\tau$ has the same property, and hence the chain homotopy giving the equality in (\ref{eq:ch-morphism-homotopy}) will again land in $\CC_\bullet^{<b}(V,\xi)_{\alpha^-,J^-,\theta^-}$, so that $\Phi(\s V,\hat\omega)$ does not depend on $\hat J$.

For the same reason, we have that the action-bounded version of the triangle (\ref{eq:ch-triangle}) commutes.

\subsubsection{Comparing $\CH^{<a}$ with different data}\label{subsubsec:comparing-cha}
Recall that in \S\ref{subsec:inter}, we consider pairs of contact forms which agree on a large set containing all of their orbits of action $<a$. We now want to show that for such forms, their action-bounded contact homology agrees.

Hence, consider two convex structures $(X^\pm,t^\pm)$ on $(V,\xi)$ and admissible pairs $(\alpha^\pm,J^\pm)$ with respect to these. Now fix some $a>0$ and some $T>0$, and suppose that $\set{t^+<T}\subset V$ contains all the Reeb orbits of $\alpha^\pm$ of action $<a$, and that $(\alpha^+,J^+)$ and $(\alpha^-,J^-)$ agree on $\set{t^+<T}$.

The two algebras $\CC_\bullet^{<a}(V,\xi)_{\alpha^+}$ and $\CC_\bullet^{<a}(V,\xi)_{\alpha^-}$ (recall that it is only the differential that depends on $J,\theta$) are then tautologically isomorphic, as they are generated by the same Reeb orbits; in fact, they are practically identical (except possibly for some set-theoretic implementation details) and we correspondingly write $\id\colon\CC_\bullet^{<a}(V,\xi)_{\alpha^+}\to\CC_\bullet^{<a}(V,\xi)_{\alpha^-}$ for this isomorphism, and for the induced isomorphism $\id\colon\CH_\bullet^{<a}(V,\xi)_{\alpha^+}\to\CH_\bullet^{<a}(V,\xi)_{\alpha^-}$.

Now, given complex structures $J^\pm$, we would like to conclude that the resulting differentials also agree, assuming -- as in the conclusion of Proposition~\ref{propn:compactness-sympl-rel} -- that $J^\pm$ agree on a large enough set so that all of the $\wh{J^+}$- and $\wh{J^-}$-holomorphic curves used in defining the differential agree.

We do indeed have such a result, though it is complicated somewhat by the ``virtual methods'' needed to define the moduli counts.

Recall the sets $\Theta$ introduced in \S\ref{subsubsec:ch-conv-setups}.
\begin{propn}\label{propn:action-bounded-complexes-agree}
  Fix $a>0$. Let $(\alpha^\pm,J^\pm)$ be $(X^\pm,t^\pm)$-admissible pairs on $(V,\xi)$, and suppose that they agree on some open set $U\subset V$ which contains all of the Reeb orbits of action $<a$ of $\alpha^\pm$, and such that all relevant $\widehat{J^\pm}$-holomorphic buildings in $\s V$ with positive end of action $<a$ are contained in $\s U\subset\s V$.

  Then for each $\theta^+\in\Theta_\rI(V,\alpha^+,J^+)$, there exists $\theta^-\in\Theta_\rI(V,\alpha^-,J^-)$ such that the map $\id\colon\CC_\bullet^{<a}(V,\xi)_{\alpha^+,J^+,\theta^+}\to\CC_\bullet^{<a}(V,\xi)_{\alpha^-,J^-,\theta^-}$ defined above is an isomorphism of complexes -- i.e., preserves the differential.
\end{propn}
\begin{proof}
  The differentials on each side are given in terms of certain virtual moduli counts $\#\Mbar^\vir_{\theta^+}$ and $\#\Mbar^\vir_{\theta^-}$ (see \cite[\S1]{pardon-contact-homology-and-vfc}). Hence, we just need to show that, given $\theta^+$, we can produce $\theta^-$ for which these moduli counts are equal for all the moduli spaces of holomorphic curves connecting the Reeb orbits of action $<a$. This follows readily from the proof in \cite[Proposition~4.34]{pardon-contact-homology-and-vfc} that the sets $\Theta_{\rI}$ are non-empty. In order to explain this, we need to recall some of the set-up from there.

  There is a category $\SSS_\rI=\SSS(V,\alpha)$ (with some extra structure), the objects of which are certain decorated trees $T$ indexing the different relevant moduli spaces of holomorphic buildings in $\s V$. There is a notion of \emph{module} over $\SSS_\rI$ and of \emph{morphism} of such modules, and the elements of the set $\Theta_\rI(V,\alpha,J)$ are morphisms of certain $\SSS_\rI$-modules associated to the data $(V,\alpha,J)$. Constructing such a morphism involves associating a certain datum to each $T\in\SSS_\rI$, and the proof in \emph{loc. cit.} that there exists a $\theta\in\Theta_\rI$ carries out this construction by an \emph{induction} on $T$.

  Now, the virtual moduli counts $\#\Mbar^\vir_\theta$ defined by $\theta$ are such that the count for the moduli space $\Mbar(T)$ indexed by $T$ is determined by the datum associated to $T$ by $\theta$. In particular, the moduli counts needed to construct $\CH_\bullet^{<a}$ are determined by the data associated to those trees $T\in\SSS_{\rI}$ that are decorated by Reeb orbits of action $<a$ (let us call such $T$ \emph{$a$-bounded}). Now there are three facts about such $T$ which will imply the desired conclusion. Let us write $\SSS_\rI^\pm$ for $\SSS(V,\alpha^\pm)$.

  The first fact is that the $\SSS_\rI^+$ and $\SSS_\rI^-$ have the same $a$-bounded objects (since $\alpha^+$ and $\alpha^-$ have the same Reeb orbits of action $<a$).

  The second fact is that the values of the relevant $\SSS_\rI^+$- and $\SSS_\rI^-$-modules on the $a$-bounded objects are the same. Briefly, the values of these modules at $T$ are given in terms of the so-called \emph{implicit atlases} associated to $T$ and to the objects $T'$ over $T$ (i.e., admitting a morphism $T'\to T$). Now, when $T$ is $a$-bounded, so is $T'$ (this is a consequence of Stokes' theorem applied to $\dv\alpha^\pm$), and the implicit atlases associated to $a$-bounded $T$ will be the same for $(V,\alpha^+,J^+)$ and $(V,\alpha^-,J^-)$: the implicit atlas is specified by a set of ``thickening data'' which only depends on $(V,\xi)$, and by ``thickened moduli spaces'', which are moduli spaces of genus 0 holomorphic buildings connecting the Reeb orbits decorating $T$, and hence for $a$-bounded $T$, will be the same for $(\alpha^+,J^+)$ and $(\alpha^-,J^-)$ by our assumption on $J^\pm$.

  The upshot of the first two facts is that it makes sense to talk about a $\theta^+\in\Theta_\rI(V,\alpha^+,J^+)$ and $\theta^-\in\Theta_\rI(V,\alpha^-,J^-)$ associating the \emph{same} data to those $T$ relevant for the definition of $\CH_\bullet^{<a}(V,\xi)_{\alpha^+,J^+,\theta^+}$ and $\CH_\bullet^{<a}(V,\xi)_{\alpha^-,J^-,\theta^-}$ and if this is done, the resulting virtual moduli counts will be the same, and hence the differentials will agree.

  The third fact is that the set of $a$-bounded objects is \emph{initial} in the partial order with respect to which the induction in \emph{loc. cit.} is performed (this follows directly from the definition of the partial order, and from the application of Stokes' theorem mentioned above).

  Hence, given $\theta^+\in\Theta_\rI(V,\alpha^+,J^+)$, we can simply define $\theta^-\in\Theta_\rI(V,\alpha^-,J^-)$ to be equal to $\theta^+$ on the $a$-bounded objects, and then proceed with the induction from there as in \emph{loc. cit}.
\end{proof}

Now, continuing in the above situation, suppose we have a second pair of $(X^\pm,t^\pm)$-admissible pairs $(\tilde\alpha^\pm,\tilde J^\pm)$ satisfying the same hypotheses (say, with respect to the same $U\subset V$), and choose $\theta^\pm\in\Theta_\rI(V,\alpha^\pm,J^\pm)$ and $\tilde\theta^\pm\in\Theta_\rI(V,\tilde\alpha^\pm,\tilde J^\pm)$ so that the two corresponding morphisms $\id$ are isomorphisms of complexes.

We now similarly want to show that, under appropriate conditions, the morphisms $\Phi(\s V,\hat{\omega})\colon\CH^{<a}_\bullet(V,\xi)_{\alpha^+,J^+,\theta^+}\to\CH^{<b}_\bullet(V,\xi)_{\tilde \alpha^+,\tilde J^+,\tilde \theta^+}$ and $\Phi(\s V,\hat{\omega})\colon\CH^{<a}_\bullet(V,\xi)_{\alpha^-,J^-,\theta^-}\to\CH^{<b}_\bullet(V,\xi)_{\tilde \alpha^-,\tilde J^-,\tilde \theta^-}$ of \S\ref{subsubsec:ch-conv-action-bounded} agree (with respect to the identifications $\id$ on both sides). We have:
\begin{propn}\label{propn:action-bounded-morphisms-agree}
  Fix $a>0$. Let $(X^\pm,t^\pm)$ be convex structures on $(V,\xi)$. Let $(\alpha^+,J^+)$ and $(\tilde\alpha^+,\tilde J^+)$ be $(X^+,t^+)$-admissible pairs, and let $(\alpha^-,J^-)$ and $(\tilde\alpha^-,\tilde J^-)$ be $(X^-,t^-)$-admissible pairs, and suppose that there is some open $U\subset V$ such that $(\alpha^\pm,J^\pm)$ and $(\tilde\alpha^\pm,\tilde J^\pm)$ each satisfy the hypotheses of Proposition~\ref{propn:action-bounded-complexes-agree} with respect to $U$ and $a$, and with respect to $U$ and some $b\ge\max(\sup(\frac{\tilde\alpha^+}{\alpha^+})a,\sup(\frac{\tilde\alpha^-}{\alpha^-})a)$, respectively. Fix $\theta^\pm$ and $\tilde\theta^\pm$ so that the instances of $\id$ in the diagram below are isomorphisms of complexes.

  Next, let $\hat J$ be an almost complex structure on $\s V$ admissible with respect to $(\alpha^+,J^+)$ and $(\tilde \alpha^+,\tilde J^+)$, and $\hat J'$ an almost complex structure on $\s V$ admissible with respect to $(\alpha^-,J^-)$ and $(\tilde \alpha^-,\tilde J^-)$.

  Suppose that $\hat{J}$ and $\hat{J}'$ are equal on $\s U\subset\s V$, and that all of the relevant $\hat{J}$-holomorphic and $\hat{J}'$-holomorphic buildings in $\s V$ with positive end of action $<a$ are contained in $\s U$.

  Then for each $\hat\theta\in\Theta_\rII(\s V,\hat\lambda,\hat J)_{(\theta^+,\theta^-)}$, there exists $\hat\theta'\in\Theta_\rII(\s V,\hat\lambda,\hat J')_{(\tilde\theta^+,\tilde\theta^-)}$ (here, the subscripts indicate that we are restricting to the fibers over the given elements under the maps $\Theta_{\rII}\to\Theta_\rI^+\times\Theta_{\rI}^-$) such that the following square commutes:
  \[
    \begin{tikzcd}
      \CC^{<a}_\bullet(V,\xi)_{\alpha^+,J^+,\theta^+}\ar[r, "\id"]
      \ar[d, "{\Phi(\s V,\hat\omega)_{\hat J,\hat\theta}}"]&
      \CC^{<a}_\bullet(V,\xi)_{\alpha^-,J^-,\theta^-}
      \ar[d, "{\Phi(\s V,\hat\omega)_{\hat J',\hat\theta'}}"]\\
      \CC^{<b}_\bullet(V,\xi)_{\tilde\alpha^+,\tilde J^+,\tilde\theta^+}\ar[r, "\id"]&
      \CC^{<b}_\bullet(V,\xi)_{\tilde\alpha^-,\tilde J^-,\tilde\theta^-}
    \end{tikzcd}
  \]
\end{propn}
\begin{proof}
  This is proven in exactly the same way as Proposition~\ref{propn:action-bounded-complexes-agree}.

  Again, the elements of the sets $\Theta_\rII$ consist of certain morphisms of $\SSS_\rII$-modules, which assign a certain datum to each decorated tree $T$. Again, we are only concerned with the ``$a$-bounded'' objects $T$, which now means those trees such that the ``incoming'' edge is decorated with a Reeb orbit of action $<a$. Again, we have that the subcatgeories of $\SSS_\rII(\s V,\alpha^\pm,\hat\lambda)$ and $\SSS_\rII(\s V,\tilde\alpha^\pm,\hat\lambda)$ on such $T$ agree, as do the values of the relevant $\SSS_\rII$-modules on these subcategories.

  And again, we have that these $T$ are initial with respect to the relevant partial order, so we can simply define $\hat\theta'$ to be equal to $\hat\theta$ on the $a$-bounded objects, and then continue the induction as in the proof of \cite[Proposition~4.34]{pardon-contact-homology-and-vfc}.
\end{proof}

\subsection{$\CH$ for convex manifolds}\label{subsec:ch-for-open}
Finally, we explain how to compare the invariants $\CH_\bullet(V,\xi;X,t)$ for different convex structures $(X,t)$, and hence show that they depend up to isomorphism only on $(V,\xi)$.

Our first task is, given convex structures $(X^\pm,t^\pm)$, admissible pairs $(\alpha^\pm,J^\pm)$, and $\theta^\pm\in\Theta_{\rI}(V,\alpha^\pm,J^\pm)$, to define a morphism
\begin{equation}\label{eq:xplus-xminus-map}
  \Phi\colon\CH_\bullet(V,\xi;X^+,t^+)_{\alpha^+,J^+,\theta^+}\to
  \CH_\bullet(V,\xi;X^-,t^-)_{\alpha^-,J^-,\theta^-}.
\end{equation}

In the following discussion, $V$ and $\xi$ will always be fixed, and for conciseness, we will just write $\sch{\alpha,J,\theta}$ and $\sch{\alpha,J,\theta}^{<a}$ in place of $\CH_\bullet(V,\xi)_{\alpha,J,\theta}$ and $\CH_\bullet^{<a}(V,\xi)_{\alpha,J,\theta}$, respectively.

To define the morphism \eqref{eq:xplus-xminus-map}, we will first define morphisms
\[
  \Phi_a\colon\sch{\alpha^+,J^+,\theta^+}^{<a}\to\sch{\alpha^-,J^-,\theta^-}
\]
and then show that these induce a map out of the colimit $\sch{\alpha^+,J^+,\theta^+}$ of the $\sch{\alpha^+,J^+,\theta^+}^{<a}$.

\begin{defn}
  A pair $(\alpha_a,J_a)$ consisting of a representing contact structure $\alpha_a$ and a $\dv\alpha_a$-compatible almost complex structure is \defword{$a$-good with respect to $\alpha^+,J^+,t^+,X^+,X^-,t^-$} (or just \defword{$a$-good}) if there exists a $T>0$ such that:
  \begin{enumerate}[(i)]
  \item\label{item:a-good-interp} $(\alpha_a,J_a)$ is $(X^-,t^-)$-compatible and agrees with $(\alpha^+,J^+)$ on $\set{t^+<T}$
  \item\label{item:a-good-orbits} $\alpha$ and $\alpha_a$ have the same Reeb orbits of action $<a$
  \item\label{item:a-good-complex} $T$ is large enough so as to satisfy the conclusion of Proposition~\ref{propn:compactness-sympl-rel} (where we take $(X^+,t^+,\alpha^+,J^+)$ for the $(X,t,\alpha,J)$.
  \item\label{item:a-good-morphism} $T$ is large enough so as to satisfy the conclusion of Proposition~\ref{propn:compactness-cob-rel} (where we take $\hat W=\s V$, and take $M_\pm=0$).
  \item\label{item:a-good-admiss} $T$ is large enough so that $\alpha^+$ is $(X^+,t^+)$-admissible with respect to $T$ (see Definition~\ref{defn:admissible}).
  \end{enumerate}
\end{defn}

\begin{rmk}\enumbelow
  \begin{enumerate}[(i)]
  \item It follows from Proposition~\ref{propn:interpolating-form} that there exist $a$-good triples.
  \item Any $b$-good triple is $a$-good for any $a<b$.
  \item If $(\alpha_a,J_a)$ is an $a$-good pair, then (by properties \ref{item:a-good-interp}-\ref{item:a-good-complex}) the hypotheses of Proposition~\ref{propn:action-bounded-complexes-agree} are satisfied, and hence there exists a $\theta_a$ such that the map $\id\colon\sch{\alpha^+,J^+,\theta^+}^{<a}\to\sch{\alpha_a,J_a,\theta_a}^{<a}$ is an isomorphism of dg-algebras. Let us then call $(\alpha_a,J_a,\theta_a)$ an \defword{$a$-good triple}.
  \end{enumerate}
\end{rmk}

Now, to define $\Phi_a$, we choose any $a$-good triple $(\alpha_a,J_a,\theta_a)$, and we then take $\Phi_a$ to be the composite
\[
  \sch{\alpha^+,J^+,\theta^+}^{<a}\tox{\id}
  \sch{\alpha_a,J_a,\theta_a}^{<a}\tox{\Phi(\s V,\hat\omega)}
  \sch{\alpha^-,J^-,\theta^-}.
\]
Here, the second morphism is defined since $(\alpha_a,J_a)$ and $(\alpha^-,J^-)$ are both $(X^-,t^-)$-admissible.

Next, we must show that this morphism is independent of the chosen triple $(\alpha_a,J_a,\theta_a)$. That is, given a second such triple $(\alpha_a',J_a',\theta_a')$, we must show that the outside of the following diagram commutes (ignore the central vertical morphism for now).
\begin{equation}\label{eq:indep-of-a-mor}
  \begin{tikzcd}
    &\sch{\alpha_a,J_a,\theta_a}^{<a}\ar[rd, "{\Phi(\s V,\hat\omega)}"]
    \ar[dd, "{\Phi(\s V,\hat\omega)}"]\\
    \sch{\alpha^+,J^+,\theta^+}^{<a}\ar[ru, "\id"]\ar[rd, "\id"']&&
    \sch{\alpha^-,J^-,\theta^-}
    \\
    &\sch{\alpha_a',J_a',\theta_a'}^{<a}\ar[ru, "{\Phi(\s V,\hat\omega)}"']
  \end{tikzcd}
\end{equation}

\begin{lem}
  For any two $a$-good pairs $(\alpha_a,J_a)$ and $(\alpha_a',J_a')$, there is a third $(\alpha_a'',J_a'')$ such that $\alpha_a''\ge\alpha_a,\alpha_a'$.
\end{lem}
\begin{proof}
  To construct $\alpha_a''$, we repeat the construction from the proof of Proposition~\ref{propn:interpolating-form}. Referring to the notation there, we can take $T^+$ to be a $T$ as in the definition of $a$-good. Then, by choosing $C$ large enough, we can ensure that $\alpha_a''=G_C\alpha$ is greater than $\alpha_a$ and $\alpha_a'$ on $\set{t^+\ge T_0}$ (since $\frac{\alpha^-}{\alpha_a}$ and $\frac{\alpha^-}{\alpha_a'}$ are bounded by Remark~\ref{rmk:admiss-rel-bounded}).

  Now, $G_C\alpha_a''$ is equal to $\alpha=\alpha_a=\alpha_a'$ on $\set{t^+\le T^+}$ and, finally, we can ensure that it is greater than $\alpha_a$ and $\alpha_a'$ on $\set{T^+\le t^+\le T_0}$ by choosing $H(t^+)$ (again in the notation from the proof of Proposition~\ref{propn:interpolating-form}) to be greater than $\frac{\alpha_a}{\alpha^+}$ and $\frac{\alpha_a'}{\alpha^+}$.
\end{proof}

Because of this lemma, it suffices to show that (\ref{eq:indep-of-a-mor}) commutes when $\alpha_a\ge\alpha_a'$. Now consider the vertical morphism in the diagram. This is defined since $(\alpha_a,J_a)$ and $(\alpha_a',J_a')$ are both $(X^-,t^-)$-admissible, and it indeed lands in the filtered piece $\set{\alpha_a',J_a',\alpha_a'}^{<a}$, as indicated, since $\alpha_a\ge\alpha_a'$ (see \S\ref{subsubsec:ch-conv-action-bounded}).

Now, it is obvious that the left triangle commutes if we take the vertical arrow to be $\id$. Hence, it suffices to show that $\id=\Phi(\s V,\hat\omega)$ (at least on homology).

We now claim that the following square commutes, which we will justify in a moment.
\[
  \begin{tikzcd}
    \sch{\alpha^+,J^+,\theta^+}^{<a}\ar[r, "\id"]
    \ar[d, "{\Phi(\s V,\hat\omega)}"]&
    \sch{\alpha_a,J_a,\theta_a}^{<a}
    \ar[d, "{\Phi(\s V,\hat\omega)}"']\\
    \sch{\alpha^+, J^+,\theta^+}^{<a}\ar[r, "\id"]&
    \sch{\alpha_a',J_a',\theta_a'}^{<a}
  \end{tikzcd}
\]
Hence, it suffices that the left morphism is the same (on homology) as $\id$. But, as in \S\ref{subsec:ch-for-fixed-convex}, the proof of \cite[Lemma~1.2]{pardon-contact-homology-and-vfc} goes through in the present (action-bounded) situation and shows that $\Phi(\s V,\hat\omega)$ is an isomorphism, whence it follows from the commutativity of \eqref{eq:ch-triangle} from \S\ref{subsec:ch-for-fixed-convex} that $\Phi(\s V,\hat\omega)\circ\Phi(\s V,\hat\omega)=\Phi(\s V,\hat\omega)$ and hence that it is equal to the identity.

Now let us see why the above square commutes.

We will apply Proposition~\ref{propn:action-bounded-morphisms-agree}, where we take the $(\alpha^+,J^+,\theta^+)$ and $(\tilde\alpha^+,\tilde J^+,\tilde\theta^+)$ in that proposition to be the present $(\alpha^+,J^+,\theta^+)$, we take $\hat J$ to be $\wh{J^+}$, we take $(\alpha^-,J^-,\theta^-)$ to be $(\alpha_a,J_a,\theta_a)$, and we take $(\tilde\alpha^-,\tilde J^-,\theta^-)=(\alpha_a',J_a',\theta_a')$.

If we can show that the hypotheses of the proposition are fulfilled for some $\hat J'$, then this will imply our claim. Now fix $T>0$ as in the definition of $a$-good, and choose any $\hat J'$ which agrees with $\wh{J^+}$ on $\set{t^+<T}$ and which is equal to $\wh{J^\pm}$ on $\set{\pm r^\pm\ge0}$, where $r^\pm$ are the symplectization coordinates of $\alpha_a$ and $\alpha_a'$ (the second condition is possible to fulfill since $\alpha_a\ge\alpha_a'$, and it is compatible with the first condition by condition \ref{item:a-good-interp} in the definition of $a$-good).

This implies, using condition \ref{item:a-good-morphism} in the definition of $a$-good, that $\hat J'$ satisfies the condition on $\hat K$ in Proposition~\ref{propn:compactness-cob-rel} (applied with $\hat W=\s V$, with the present $(\alpha^\pm,J^\pm)$, with $M^\pm=0$). This in turn implies that the hypotheses of Proposition~\ref{propn:action-bounded-morphisms-agree} are fulfilled, as desired

This concludes the proof that the above square commutes, and hence that the triangle on the right of \eqref{eq:final-triangle-for-ch} commutes.

Thus, we have defined a canonical morphism $\Phi_a\colon\sch{\alpha^+,J^+,\theta^+}^a\to\sch{\alpha^-,J^-,\theta^-}$. Next, to see that these induce a morphism $\Phi\colon\sch{\alpha^+,J^+,\theta^+}\to\sch{\alpha^-,J^-,\theta^-}$, we must show for each $a<b$ that the triangle
\[
  \begin{tikzcd}[row sep=0pt]
    \sch{\alpha^+,J^+,\theta^+}^{<a}\ar[rd, "\Phi_a"]\ar[dd]\\
    &\sch{\alpha^-,J^-,\theta^-}\\
    \sch{\alpha^+,J^+,\theta^+}^{<b}\ar[ru, "\Phi_b"']
  \end{tikzcd}
\]
commutes. However, noting that any $b$-good triple is also $a$-good, this amount to showing that the outside of the following diagram commutes.
\[
  \begin{tikzcd}
    \sch{\alpha^+,J^+,\theta^+}^{<a}\ar[dd]\ar[r, "\id"]&
    \sch{\alpha_b,J_b,\theta_b}^{<a}\ar[rd, "{\Phi(\s V,\hat\omega)}"]\ar[dd]
    \\
    &&\sch{\alpha^-,J^-,\theta^-}\\
    \sch{\alpha^+,J^+,\theta^+}^{<b}\ar[r, "\id"]&
    \sch{\alpha_b,J_b,\theta_b}^{<b}\ar[ru, "{\Phi(\s V,\hat\omega)}"']
  \end{tikzcd}
\]
However, the commutativity of both the square and the triangle are trivial.

Hence, we now have a canonical morphism $\sch{\alpha^+,J^+,\theta^+}\to\sch{\alpha^-,J^-,\theta^-}$.

Finally, to obtain an invariant only depending on $V$, we must show, given $(\alpha^k,J^k,\theta^k)$ for $k=0,1,2$, that the triangle on the right of
\begin{equation}\label{eq:final-triangle-for-ch}
  \begin{tikzcd}
    \sch{\alpha^0,J^0,\theta^0}^{<a}\ar[r]\ar[d, "\Phi_a"]
    \ar[dd, "\Phi_a", bend right=50pt, out=-90, in=-90]&
    \sch{\alpha^0,J^0,\theta^0}\ar[d, "\Phi"]
    \ar[dd, "\Phi", bend left=50pt, out=90, in=90]\\
    \sch{\alpha^1,J^1,\theta^1}^{<b}\ar[r]\ar[d, "\Phi_b"]&
    \sch{\alpha^1,J^1,\theta^1}\ar[d, "\Phi"]\\
    \sch{\alpha^2,J^2,\theta^2}^{<c}\ar[r]&
    \sch{\alpha^2,J^2,\theta^2}
  \end{tikzcd}
\end{equation}
commutes. By abstract nonsense (using that $\Phi$ is defined as the induced map out of a colimit), it suffices to show that the triangle on the left commutes, i.e., that the outside of the following diagram commutes (where $b$ and $c$ are chosen large enough so that the images of $\Phi_a$ and $\Phi_b$ are contained in the corresponding filtered pieces, as shown).
\[
  \begin{tikzcd}[column sep=-15pt]
    \sch{\alpha^0,J^0,\theta^0}^{<a}\ar[rr, "\id"]\ar[rd, "\id"]&&
    \sch{\alpha^{02}_a,J^{02}_a,\theta^{02}_a}^{<a}
    \ar[rr, "{\Phi(\s V,\hat\omega)}"']\ar[rd, "{\Phi(\s V,\hat\omega)}"']
    &&\sch{\alpha^2,J^2,\theta^2}^{<c}\\
    &\sch{\alpha^{01}_a,J^{01}_a,\theta^{01}_a}^{<a}\ar[ru, "\id"]\ar[rd, "{\Phi(\s V,\hat\omega)}"']&&
    \sch{\alpha^{12}_b,J^{12}_b,\theta^{12}_b}^{<b}\ar[ru, "{\Phi(\s V,\hat\omega)}"']
    \\
    &&\sch{\alpha^1,J^1,\theta^1}^{<b}\ar[ru, "\id"]
  \end{tikzcd}
\]
Here, each $\sch{\alpha^{ij}_a,J^{ij}_a,\theta^{ij}_a}$ is an $a$-good triple, and $\sch{\alpha^{12}_b,J^{12}_b,\theta^{12}_b}$ is a $b$-good triple. The triangle on the left commutes trivially and the triangle on the right is (\ref{eq:ch-triangle}) from \S\ref{subsubsec:defining-CH-conv}. It remains to consider the square, for which we use our freedom in choosing the $(\alpha^{ij}_a,J^{ij}_a)$ and the complex structures $\hat J$ and $\hat J'$ used to define (respectively) the lower-left and upper-right morphisms $\Phi(\s V,\hat\omega)=\Phi(\s V,\hat\omega)_{\hat J,\hat\theta}$ and $\Phi(\s V,\hat\omega)=\Phi(\s V,\hat\omega)_{\hat J',\hat\theta'}$.

We will carefully make these choices in a certain order, and will be choosing various positive constants $0<T_0<T_1<\ldots<T_4$ along the way (see Figure~\ref{fig:01-and-12}).

We first choose any $a$-good pair $(\alpha_a^{01},J_a^{01})$. Hence, there are $T_1>T_0>0$ such that $(\alpha_a^{01},J_a^{01})$ is equal to $(\alpha^0,J^0)$ for $\set{t^+<T_0}$ and equal to $(\alpha^1,J^1)$ for $\set{t^+>T_1}$.

Next, choose $T_2>T_1$ and take $\hat J$ to be an admissible (with respect to $\hat J^{01}_a$ and $J^1$) almost complex structure on $\s V$ which is equal to $\wh{J^1}$ on $\set{t^+>T_2}$.

Now fix some $M_\pm\in\R$ such that $\hat J$ is equal, respectively, to $\wh{J^{01}_a}$ and $\wh{J^1}$ on $\set{r_+\ge M_+}$ and $\set{r_-\le -M_-}$ (where $r_\pm$ are the symplectization coordinates induced by $\alpha^{01}_a$ and $\alpha^1$), and let $T_3>T_2$ be a $T$ as in Proposition~\ref{propn:compactness-cob-rel}, applied with $\hat W=\s V$ and with the present $\hat J$.

Next, we can choose $(\alpha_a^{12},J_a^{12})$ so that it agrees with $(\alpha^1,J^1)$ on $\set{t^+<T_4}$ for some $T_4>T_3$.
\begin{figure}
  \centering
  \begin{tikzpicture}
    \node[left] at (0,5) {$(\alpha^{01}_a,J^{01}_a)$};
    \draw (0,5) -- (10,5); \draw (0,5+0.25) -- (0,5-0.25);
    \node[above] at (2.5,5) {$_{T_0}$}; \draw(2.5,5+0.125) -- (2.5,5-0.125);
    \node[above] at (1.25,5) {$_{(\alpha^0,J^0)}$};
    \node[above] at (3.5,5) {$_{T_1}$}; \draw(3.5,5+0.125) -- (3.5,5-0.125);
    \node[above] at (6,5) {$_{(\alpha^1,J^1)}$};
    \node[left] at (0,4) {$\hat J$};
    \draw (0,4) -- (10,4); \draw (0,4+0.25) -- (0,4-0.25);
    \node[above] at (4,4) {$_{T_2}$}; \draw(4,4+0.125) -- (4,4-0.125);
    \node[above] at (6,4) {$_{\wh{J^1}}$};
    \node[left] at (0,3) {Prop~\ref{propn:compactness-cob-rel}};
    \draw (0,3) -- (10,3); \draw (0,3+0.25) -- (0,3-0.25);
    \node[above] at (4.5,3) {$_{T_3}$}; \draw(4.5,3+0.125) -- (4.5,3-0.125);
    \node[left] at (0,2) {$(\alpha^{12}_a,J^{12}_a)$};
    \draw (0,2) -- (10,2); \draw (0,2+0.25) -- (0,2-0.25);
    \node[above] at (5.5,2) {$_{T_4}$}; \draw(5.5,2+0.125) -- (5.5,2-0.125);
    \node[above] at (2.25,2) {$_{(\alpha^1,J^1)}$};
    \draw(6.5,2+0.125) -- (6.5,2-0.125);
    \node[above] at (8.75,2) {$_{(\alpha^2,J^2)}$};
    \node[left] at (0,1) {$(\alpha^{02}_a,J^{02}_a)$};
    \draw (0,1) -- (10,1); \draw (0,1+0.25) -- (0,1-0.25);
    \node[above] at (4.5,1) {$_{T_3}$}; \draw(4.5,1+0.125) -- (4.5,1-0.125);
    \node[above] at (2.25,1) {$_{(\alpha^{01},J^{01})}$};
    \node[above] at (7,1) {$_{(\alpha^{12},J^{12})}$};
    \node[left] at (0,0) {$\hat J'$};
    \draw (0,0) -- (10,0); \draw (0,0+0.25) -- (0,0-0.25);
    \node[above] at (4.5,0) {$_{T_3}$}; \draw(4.5,0+0.125) -- (4.5,0-0.125);
    \node[above] at (2.25,0) {$_{\hat J}$};
    \node[above] at (7,0) {$_{\widehat{J^{12}}}$};
  \end{tikzpicture}
  \caption{The meanings and relative locations of the $T_i$}
  \label{fig:01-and-12}
\end{figure}
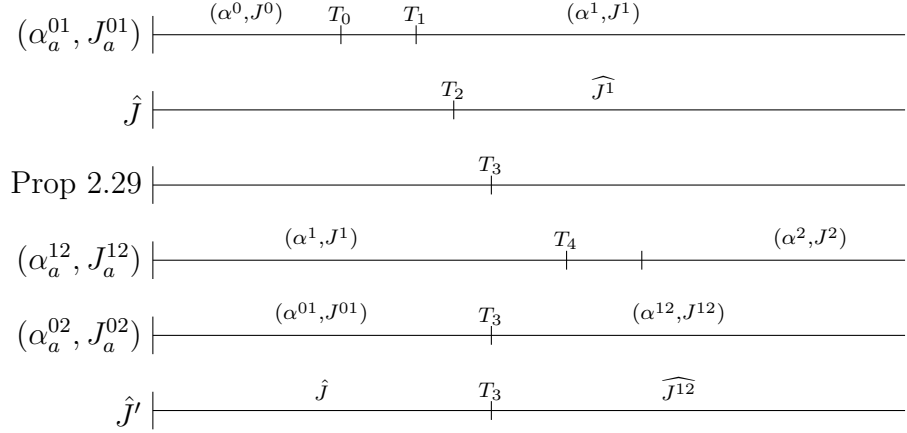

Note that $(\alpha^{01}_a,J^{01}_a)$ and $(\alpha^{12}_a,J^{12}_a)$ are both equal to $(\alpha^1,J^1)$ on $\set{T_1<t^+<T_4}$. Hence, since $T_1<T_3<T_4$, we can define $(\alpha^{02}_a,J^{02}_a)$ to be equal to $(\alpha^{01}_a,J^{01}_a)$ on $\set{t^+\le T_3}$ and to $(\alpha^{12}_a,J^{12}_a)$ on $\set{t^+\ge T_3}$. It follows that $(\alpha^{02}_a,J^{02}_a)$ is $a$-good, assuming that we chose $T_3$ large enough (which we were free to do).

Finally, for the same reason, we can define $\hat J'$ to be equal to $\hat J$ on $\set{t^+\le T_3}$ and to $\wh{J^{12}}$ on $\set{t^+\ge T_3}$, and it is then easy to see that it is admissible with respect to $J^{02}_a$ and $J^{12}_b$, and specifically that it is equal, respectively, to $\wh{J^{02}_a}$ and $\wh{J^{12}_b}$ on $\set{\tilde r^+\ge M^+}$ and $\set{\tilde r^-\le -M^-}$ (where $\tilde r^\pm$ are the symplectization coordinates induced by $\alpha^{02}_a$ and $\alpha^{12}_b$).

It then follows that $\hat J'$ satisfies the hypotheses on $\hat K$ in Proposition~\ref{propn:compactness-cob-rel} (applied with the present $\hat J$ and with $T=T_3$) and hence that the hypotheses of Proposition~\ref{propn:action-bounded-morphisms-agree} (applied to the present $\hat J$ and $\hat J'$) are satisfied, and hence that the square in question commutes.

This establishes the commutativity of the triangle on the right of \eqref{eq:final-triangle-for-ch}, and we thus obtain a commuting diagram consisting of all the $\CH_\bullet(V,\xi;X,t)_{\alpha,J,\theta}$ and the canonical isomorphisms between them.

\begin{defn}
  We define $\CH_\bullet(V,\xi)$ to be the colimit  of the diagram consisting of all the $\CH_\bullet(V,\xi;X,t)_{\alpha,J,\theta}$ together with the above canonical isomorphisms.
  (In particular, $\CH_\bullet(V,\xi;X,t)$ is canonically isomorphic to each of the $\CH_\bullet(V,\xi;X,t)_{\alpha,J,\theta}$ appearing in the diagram.\footnote{%
    At the risk of dwelling on set-theoretic trivialities, but in the hopes of clarifying the situation, we note that since ``the'' colimit is only defined up to isomorphism, we could (using the axiom of choice) simply define $\CH_\bullet(V,\xi)$ for each $(V,\xi)$ to be equal to an arbitrarily chosen $\CH_\bullet(V,\xi;X,t)_{\alpha,J,\theta}$.
    A more canonical construction of the colimit would be as a quotient of the disjoint union of all the $\CH_\bullet(V,\xi;X,t)_{\alpha,J,\theta}$.%
  })
\end{defn}

\subsection{Relationship to sutured contact homology}\label{subsec:sutured}
Let $(\overline V,\xi)$ be a compact contact manifold with convex boundary $\partial\overline V$, and let $V$ be its interior.

We can now associate to $(\overline V,\xi)$ the contact homology $\CH(V,\xi)$ of the interior, as just defined.
On the other hand, as mentioned in the introduction, in \cite{colin-others-sutures-and-ch}, another contact homology group is associated to this situation; namely, to $(\overline V,\xi)$ is associated a certain \emph{sutured contact manifold} $(V_\sut,\Gamma,U(\Gamma),\xi_\sut)$ and to this is associated the \emph{sutured contact homology} $\CH(V_\sut,\Gamma,\xi_\sut)$.

Both of these invariants are computed as the contact homology of some contact form on $V$, which suggests that they might be isomorphic; this would show that the sutured contact homology, which a priori depends on the sutured structure of $V_\sut$ is in fact an invariant of the interior $V$.
However, there are obstacles to establishing this isomorphism, as we now explain.

Suppose $(V_\sut,\Gamma,U(\Gamma),\xi_\sut)$ is a sutured contact manifold corresponding to $\overline V$ (see \cite{colin-others-sutures-and-ch} for the definition and for the notation we will be using here).
Then $\overline V$ can be recovered by rounding the corners of $V_\sut$, i.e., we have a smooth embedding $\overline V\hto V_\sut$ whose complement is a small neighbourhood of the corners.

Now the sutured contact homology of $V_\sut$ is defined as the contact homology of a certain completion $(V_\sut^{\alpha,*},\alpha^*)$, where $\alpha$ is an adapted contact form (in the sense of \cite[Definition~2.8]{colin-others-sutures-and-ch}).
This is obtained by attaching ``vertical'' ends $R_+(\Gamma)\times[1,\infty)_t$ and $R_-(\Gamma)\times(-\infty,-1]_t$ to the ``top and bottom'' $R_\pm(\Gamma)$ of $V_\sut$ and then attaching a ``horizontal'' end $\Gamma\times\R_t\times[0,\infty)_\tau$.
The completed contact form $\alpha^*$ then has the form $C\dt+e^\tau\beta_0$ on the horizontal end, for some contact form $\beta_0$ on $\Gamma$, and has the form $C\dt+\beta_\pm$ on the vertical ends for some Liouville forms $\beta_\pm$ on $R_\pm(\Gamma)$ (so that $\beta_\pm=e^\tau\beta_0$ in the region where $\tau$ is defined); see Figure~\ref{fig:sutured-completion}.

  \begin{figure}\label{fig:sutured-completion}
    \centering
    \begin{tikzpicture}[yscale=0.5]
      \path
        (3,-2) node {$\alpha^*=C\dt-\beta_-$}
        (3,2) node {$\alpha^*=C\dt+\beta_+$}
        (8,0) node {$\alpha^*=C\dt+e^\tau\beta_0$}
        (3,0) node {$\operatorname{int}V_\sut$}
        (-0.8,1) node {$R_+(\Gamma)$}
        (-0.8,-1) node {$R_-(\Gamma)$};
      \fill (6,0) ellipse [x radius=2pt, y radius=4pt] node[anchor=east] {$\Gamma$};
      \draw (0,-1) -- (6,-1) -- (6,1) -- (0,1);
      \draw[dashed]
        (6,1) -- (6,3)
        (6,-1) -- (6,-3);
      \draw[->] (5.5,-3.5) -- (8,-3.5) node[anchor=west] {$\tau$};
      \draw[->] (-2,-4) -- (-2,4) node[anchor=west] {$t$};

      \draw
        (6,-3.3) -- (6,-3.7) node [anchor=north] {$0$}
        (-2.1,1) node [anchor=east] {$1$} -- (-1.9,1)
        (-2.1,-1) node [anchor=east] {$-1$} -- (-1.9,-1);
    \end{tikzpicture}
    \caption{The completion of a sutured contact manifold}
\end{figure}
To compare this with $\CH(V,\xi)$ as defined in this paper, we first need to find a contactomorphism $V_\sut^*\to V$.
By Remark~\ref{rmk:convexity-definitions}, we have that $V$ is isomorphic to the contact manifold obtained by attaching an infinite end $\partial\overline V\times[0,\infty)$ to $\overline V$ using a transverse contact vector field.
Hence, it suffices to find a complete contact vector field on $V_\sut$ which is transverse to $\partial\overline V\hto V_\sut$.

We can do this, for example, by taking a contact Hamiltonian (see Remark~\ref{rmk:convexity-definitions}) $H=H(t)$ only depending on $t$ and satisfying $H(t)=\pm1$ for $\pm t>1-\varepsilon$; one then computes that the associated contact vector field $X_H$ is given by $H\partial_t+H'\partial_\tau$ for $\tau>-\varepsilon$, and by $\partial_t$ for $\pm t>1-\varepsilon$.

Unfortunately, with this convex structure, the form $\alpha^*$ is not admissible; since it does not converge to a contact form as $\tau\to\infty$.
It may be possible to remedy this -- that is, to render $\alpha^*$ admissible -- by the choice of a different convex structure on $V_\sut^*$, but it is not obvious how to do this.

A different approach would be to extend either the definition of sutured contact homology so as to show that it can be computed by forms admissible in our sense, or to extend our definition of contact homology so as to show it can be computed by the form $\alpha^*$.

However, this paper and \cite{colin-others-sutures-and-ch} each use different compactness arguments depending on the particular kind of contact forms that are admitted, and so it is not obvious how this is to be done.

\section{Exotic contact structures}\label{sec:exotic-contact-strucs}
We now apply our invariant to the prove the existence of exotic contact structures on $\R^{2n-1}$ for all $n>2$; we will show that the standard contact structure on $\R^{2n-1}$ has trivial contact homology, and then produce other contact structures for which this is not the case.

The argument will be slightly different for $n$ odd and $n$ even -- i.e., for $2n-1$ of the form $4m+1$ and $4m+3$, respectively.

In both cases, we will begin with a contact structure on a ``Brieskorn manifold'' $M$ and then puncture $M$ to obtain the desired contact structure. In the $4m+1$ case, the contact structure on $M$ was already introduced in \cite{ustilovsky-infinitely-many-spheres}. The $4m+3$ case was not considered in \emph{op. cit.}, but we show that the arguments there can be modified to produce a suitable contact form in this case as well. We note that in this case, the Brieskorn manifold in question is actually a (smoothly) exotic sphere, but that suffices for our purposes, since after puncturing, we still obtain a (smoothly) standard $\R^{4m+3}$ (on which there are no exotic smooth structures).

In both cases, after puncturing (or, more precisely, performing a contact connected sum with $\R^{2n-1}_\std$), we obtain a contact structure on $\R^{2n-1}$. A little bit of careful bookkeeping will then allow us to show that the contact homology of this contact structure is nontrivial.

Finally, by repeating the argument with $M$ replaced by a connected sum of several copies of $M$, we will produce infinitely many contact structures on $\R^{2n-1}$ with pairwise distinct contact homology.

\subsection{Computation for standard $\R^{2n-1}$}
We now would like to show that for all $n>1$, $\CH_\bullet(\R^{2n-1}, \xi_\std)$ is trivial (i.e., isomorphic to $\Q$).

If the standard contact form $\alpha_\std=\dz+\frac12\sum_{i=1}^{n-1}x_i\dy_i-y_i\dx_i$ were admissible, this would be immediate, as this form has no Reeb orbits at all. Hence, our task is to find an admissible form which still has no Reeb orbits.

First, let us see that $\R^{2n-1}_\std$ admits a convex structure.
We let $X$ be the complete vector field $2z\partial_z+\sum_{i=1}^{n-1}x_i\partial x_i+y_i\partial y_i$ and we let $t\colon\R^{2n-1}-\set{0}\to\R$ be the function $t=\frac14\log(z^2+\sum_ix_i^4+y_i^4)$.
We then have $X\intprod\dt=1$.

We now compute:
\[
  \begin{split}
    \Ldv_X\alpha_\std&=
    X\intprod\dv\alpha_\std+\dv(X\intprod\alpha_\std)=\\&=
    \plr{\sum_{i=1}^{n-1}x_i\dy_i-y_i\dx_i}+
    \dv\plr{2z+\frac12\sum_{i=1}^{n-1}x_iy_i-y_ix_i}=2\alpha_\std
  \end{split}
\]
showing that $X$ is a contact vector field, and hence $(X,t)$ is a convex structure (after cutting off $t$ near the origin so that it extends from $\R^{2n-1}-\set{0}$ to $\R^{2n-1}$).

\begin{lem}\label{lem:Rn-no-reeb-orbits}
	There exists an admissible contact form on $(\R^{2n-1}, \xi_\std, X, t)$ which doesn't have any Reeb orbits.
\end{lem}
\begin{proof}
  It suffices to construct an admissible contact form $\tilde\alpha$ whose Reeb vector field satisfies $\Rb_{\tilde\alpha}\intprod\dz>0$.

  Note that the form $\alpha=e^{-2t}\alpha_{\std}$ on $\R^{2n-1}-\set{0}$ is $t$-invariant (i.e., $X$-invariant).
  Setting $u\defeq X\intprod\alpha=2z/e^{2t}$ and $\beta=\frac1{2e^{2t}}\sum_ix_i\dy_i-y_i\dx_i$, we have
  \[
    \alpha=\beta+\tfrac12\du+u\dt,
  \]
  where $\partial_t\intprod\du=\partial_t\intprod\beta=0$ (and where $\partial_t=X$).

  Let us write $s$ for the restriction of $u=2z/e^{2t}$ to the complement of the $z$-axis $U=\set{\abs{\vec x}+\abs{\vec y}>0}$.
  On this set, we have a coordinate system $\br{\pi,s,t}\colon U\toi S\times(-2,2)\times\R$, where $S=\set{(\vec x,\vec y)\in\R^{2n-2}\mid\abs{\vec x}^4+\abs{\vec y}^4=1}$, and where $\pi(\vec x,\vec y,z)=e^{-t}(1-\frac{s^2}4)^{-1/4}(\vec x,\vec y)$.
  Then $\alpha$ takes the form
  \[
    \rstr{\alpha}{U}=(1-\tfrac{s^2}{4})^{1/2}\beta_0+\tfrac12\ds+s\dt,
  \]
  where $\beta_0=\pi^*\beta_0$ is a form pulled back from a contact form on $S$ (namely $\beta_0=\frac12(\sum_ix_i\dy_i-y_i\dx_i)\in\Omega^1(S)$).
  Note that this is similar to the contact forms appearing in Lemma~\ref{lem:normal-form-near-gamma}~and~\ref{lem:alpha-phi-reeb-field}, though not identical.
  The rest of the construction will parallel what we did in Proposition~\ref{propn:admissible-forms-exist}.

  Now fix some $T_0>0$ and let $\phi(s,t)$ be a positive function such that
  \begin{itemize}
  \item $\phi=e^{2t}$ for $t<T_0$.
  \item $\phi$ is bounded and the function $\phi_{\infty}(s)=\lim_{t\to\infty}\phi(s,t)$ is smooth.
  \item $\phi_t>0$ and $s\phi_s\le0$.
  \end{itemize}
  (We can find such $\phi$ of the form $\phi(s,t)=\rho_1(t)+\rho_2(s)\rho_3(t)$.)

  We now claim that $\tilde\alpha=\phi\alpha$ is as desired (namely, is admissible and satisfies $\Rb_{\tilde\alpha}\intprod\dz>0$), where we regard $\phi=\phi(u,t)$ as a function on $\R^{2n-1}-\set{0}$, and where $\tilde\alpha$ extends over the origin since $\phi\alpha=\alpha_\std$ near the origin.

  Note first that condition~\ref{item:admissible-nondegen} (non-degeneracy of $\tilde\alpha$) will follow from $\Rb_{\tilde\alpha}$ not having any closed orbits, that condition~\ref{item:admissible-converge} (convergence as $t\to\infty$) holds since $\tilde\alpha$ converges to $\phi_\infty\alpha$, and that condition~\ref{item:admissible-scale} will follows automatically from \ref{item:admissible-s}~and~\ref{item:admissible-t}, since if $g(t)$ is a bounded, non-decreasing function, then the function $\rstr{g\cdot\phi}{\set{t>T_0}}$ still satisfies the above conditions on $\phi$.

  Next, in the region $\set{t<T_0}$, we only need to check $\Rb_{\tilde\alpha}\intprod\dz>0$ (since admissibility only concerns the behaviour of $\tilde\alpha$ for large $t$), but here we have $\tilde\alpha=\alpha_\std$ and hence $\Rb_{\tilde\alpha}=\partial_z$.

  Next, we consider the points $p\ne 0$ on the $z$-axis $\set{\vec{x}=\vec{y}=0}$.
  Here, we have $\tilde\alpha=\phi\beta+\frac12\phi\du+\phi u\dt$.
  Note that at such a point $p$, we have $\ker\du=\ker\dt$ and $\beta=0$, and hence $\dv\tilde\alpha=\phi\dv\beta$.
  It follows that $\Rb_{\tilde\alpha}(p)=(u\phi)\I\partial_t=e^{2t}\phi\I\partial_z$, and hence that $\Rb_{\tilde\alpha}\intprod\dz>0$ and also that condition \ref{item:admissible-s} of admissibility holds.
  Condition \ref{item:admissible-t} holds vacuously here.

  Finally, let us consider points off of the $z$-axis.
  Here, setting $Q=(1-\frac{s^2}{4})^{1/2}$, so that $\dv Q=-\frac{s}{4Q}\ds$, we have $\tilde\alpha=\phi Q\beta_0+\frac12\phi\ds+\phi s\dt$, and hence
  \[
    \dv\tilde\alpha=
    \plr{\phi_s Q-\phi\frac{s}{4Q}}\ds\wedge\beta_0+
    \phi_tQ\,\dt\wedge\beta_0+
    \phi Q\, \dv\beta_0+
    (\phi+s\phi_s-\tfrac12\phi_t)\ds\wedge\dt.
  \]
  A similar computation to the one in Lemma~\ref{lem:alpha-phi-reeb-field} shows that
  \[
    \Rb_{\tilde\alpha}=
    \tfrac{Q^2}{\phi^2}
    \plr{
      Q\I (\phi+s\phi_s-\tfrac12\phi_t)\Rb_{\beta_0}+
      \phi_t\partial_s-
      (\phi_s-\phi\tfrac{s}{4Q^2})\partial_t
    }.
  \]
  We thus have
  $\Rb_{\tilde\alpha}\intprod\dz=
  e^{2t}\Rb_{\tilde\alpha}\intprod(\frac12\ds+s\dt)=
  \tfrac{e^{2t}Q^2}{\phi^2}
  \plr{\frac12\phi_t-s\phi_s+\frac{s^2}{4Q^2}\phi}>0$.

  Finally, condition~\ref{item:admissible-s} of admissibility holds since $s\Rb_{\tilde\alpha}\intprod\dt=\tfrac{Q^2}{\phi^2}(\phi\frac{s^2}{4Q^2}-s\phi_s)$, which is positive when $s\ne0$, and condition~\ref{item:admissible-t} holds since $\Rb_{\tilde\alpha}\intprod\ds=\tfrac{Q^2}{\phi^2}\phi_t>0$.
\end{proof}

\begin{propn}\label{propn:Rn-triv-contact-homology}
  $\CH_\bullet(\R^{2n-1},\xi_\std)\cong\Q$.
\end{propn}
\begin{proof}
  This is immediate from Lemma~\ref{lem:Rn-no-reeb-orbits} and the description of $\CC_\bullet$ in terms of Reeb orbits (see Remark~\ref{rmk:ch-details}).
\end{proof}

\subsection{Brieskorn manifolds}
In what follows, we will need to make use of the grading on $\CH_\bullet$: we recall that a null-homologous Reeb orbit $\gamma$ in a $(2n-1)$-dimensional contact manifold $V$ has a well-defined \emph{Conley-Zehnder index}, and its \emph{homological degree} is defined to be its Conley-Zehnder index plus $(n-3)$. If the homology groups $\Hm_1(V)$ and $\Hm_2(V)$ vanish, this induces a grading on $\CC_\bullet(V,\xi)$, where the degree of a generator given by a word of Reeb orbits is equal to the sum of the homological degrees of the individual Reeb orbits (see \cite[\S\S1.8,2.12]{pardon-contact-homology-and-vfc}).

We now have from \cite{ustilovsky-infinitely-many-spheres}:
\begin{propn}\label{propn:ustilovsky-sphere-degrees-1mod4}
  For each $m>0$, there exists a contact structure $\xi$ on $\s^{4m+1}=\s^{2n-1}$ (where $n=2m+1$) and a representing non-degenerate contact form, all of whose Reeb orbits have even homological degree and such that there are no orbits of homological degree $k<2n-4$, there is one orbit of homological degree $k$ for each even $k=2n-4,\ldots,4n-10$, and there are two orbits of homological degree $4n-8$.
\end{propn}
\begin{proof}
  This follows from \cite[Lemma~4.3]{ustilovsky-infinitely-many-spheres} with any $p>3$ (by considering the case $N=1$).
\end{proof}

We have an analogous result in dimension $4m+3$, but it is not quite as readily available from the work in \cite{ustilovsky-infinitely-many-spheres}:
\begin{propn}\label{propn:ustilovsky-sphere-degrees-3mod4}
  For each $m>0$, there exists a contact manifold $(M,\xi)$ with a non-degenerate contact form $\alpha$ such that $M$ is homeomorphic (but not necessarily diffeomorphic) to $\s^{4m+3}=\s^{2d-1}$ (where $d=2m+2$), such that all the Reeb orbits of $\alpha$ have even homological degree, and such that there are no orbits of homological degree $k<2d-4$, there is one orbit of homological degree $k$ for each $k=2d-4,\ldots,4d-12$, and there are two orbits of homological degree $4d-10$.
\end{propn}
\begin{proof}
  This will be obtained by repeating the calculations done in \cite{ustilovsky-infinitely-many-spheres}, with some modifications. Let us set $n=d-1$.

  The manifold $M$ is again a Brieskorn manifold $\Sigma(a)=\Sigma(a_{-1},a_0,\ldots,a_n)$ as in \cite{ustilovsky-infinitely-many-spheres} (the unusual indexing is to ease comparison with \emph{op. cit.}), where now $a_{-1}=q$, $a_0=p$, $a_1=\cdots=a_n=2$, where $p$ and $q$ are distinct odd primes.
  According to \cite[Satz~1]{brieskorn-beispiele-von-sing}, $\Sigma(a)$ is indeed homeomorphic to $\s^{2d-1}$.

  We take the same contact form $\alpha=\frac{i}8\sum_{j=-1}^na_j(z_jd\bar z_j-\bar z_jdz_j)$ as in \cite{ustilovsky-infinitely-many-spheres},\footnote{We would like to take the opportunity to clarify a confusing circumstance regarding this form; in \emph{op. cit.}, the fact that this is a contact form is attributed to \cite{lutz-meckert-brieskorn}. However, in the latter paper, the form is defined with $a_j\I$ in place of $a_j$. This is the wrong form; later in that paper, it becomes clear that they really meant to use the form with $a_j$ and not $a_j\I$. However, in \cite{ustilovsky-infinitely-many-spheres} and, as far as the present authors could tell, in all of the other papers that cite \cite{lutz-meckert-brieskorn}, the correct form is used without any mention of the error.} where $z_{-1},\ldots,z_n$ are the coordinates on $\C^{n+3}$.

  Now, referring to \cite[\S4]{ustilovsky-infinitely-many-spheres}, we introduce the coordinates $w_0,\ldots,w_n$ in the same way, and now also define $w_{-1}=z_{-1}$.

  Next, we define $H(w)$ in the same way as in \emph{loc. cit.}, except that we add an additional term $(\varepsilon_{-1}-1)\abs{w_{-1}}^2$ with $0<\varepsilon_{-1}<1$. This is again a real positive function on $M$, and we again define $\alpha'=H\I\alpha$. The Reeb vector field of $\alpha'$ is then given as in \cite[Lemma~4.1]{ustilovsky-infinitely-many-spheres}, except that there is now an additional entry of $\varepsilon_{-1}\frac{4i}{q}w_{-1}$ at the beginning.

  Similarly, the expression for the flow $\phi^t(w)$ is the same as given on \cite[p.~786]{ustilovsky-infinitely-many-spheres}, with an additional entry $e^{4it\varepsilon_{-1}/q}w_{-1}$ at the beginning.

  The Reeb orbits are then exactly the same ones $\gamma_0,\gamma_j^+,\gamma_j^-$ given there (though each with an additional entry 0 at the beginning) assuming that the $\varepsilon_j$ (including $\varepsilon_{-1}$!) are irrational and linearly independent over $\Q$.

  The argument in \cite[Lemma~4.2]{ustilovsky-infinitely-many-spheres} that these Reeb orbits are non-degenerate remains valid, but there are some slight changes in the computation of the Conley-Zehnder indices.

  First, in the expression for $\xi_w$, the first equation has an additional term $qw_{-1}^{q-1}w_{-1}v_{-1}$ on the left hand side. Next, we define the vectors $X_1$ and $X_2$ in the same way but with an additional first entry $\bar w_{-1}^{q-1}$ and $-2iw_{-1}/q$, respectively, and we keep the definitions of $Y_1$ and $Y_2$. The expressions for $\omega(X_1,Y_2)$, $\omega(Y_1,Y_2)$, and $\omega(X_1,Y_1)$ then have an additional term $\frac{q-2}2\operatorname{Im}(w_{-1}^q)$, $-\frac{q-2}2\operatorname{Re}(w_{-1}^q)$, and $\frac{q}{2}\abs{w_{-1}}^{2q-2}$, respectively, and the final expression for $\tilde Y_2$ has an additional term $-\frac{q-2}{2}\frac{w_{-1}^q}{\omega(X_1,Y_1)}X_1$.

  Next, the expressions for $\phi^t_*\tilde X_1(w)$ and so on are the same, except that $\phi^t_*\tilde Y_2(w)$ has an extra term $e^{4it}(e^{\varepsilon_{-1}}-1)\frac{q-2}{2}\frac{w_{-1}^q}{\omega(X_1(w),Y_1(w))}X_1(w)$; however, since $w_{-1}$ vanishes along all of the Reeb orbits $\gamma_0,\gamma^+_j,\gamma^-_j$, this term does not affect any of the further computations.

  The formula for $\mu(N\gamma_0)$ then has an additional term $2\floor{2N\frac pq\varepsilon_{-1}}+1$, and the formulas for $\mu(N\gamma^\pm_{j})$ have an additional term $2\floor{\frac{2N\varepsilon_{-1}}{q(1\pm\varepsilon_{j})}}+1$. Hence, recalling that now $d=n+1$, so that we should add $d-3=n-2$ to the Conley-Zehnder index to obtain the homological degree, we have an additional term of $2\floor{2N\frac pq\varepsilon_{-1}}+2$ and $2\floor{\frac{2N\varepsilon_{-1}}{q(1\pm\varepsilon_{j})}}+2$, respectively, in the expressions for $\tilde\mu(N\gamma_0)$ and $\tilde\mu(N\gamma_j^\pm)$ in the statement of \cite[Lemma~4.2]{ustilovsky-infinitely-many-spheres}; in particular, the homological degrees are still all even.

  Hence, if we fix $\varepsilon_1,\ldots,\varepsilon_n$, then for any $K>0$, by choosing $\varepsilon_{-1}$ small enough, we conclude that the number of orbits of homological degree $k$ is given by the number $c_{k-2}$ from \cite[Lemma~4.3]{ustilovsky-infinitely-many-spheres} for $k<K$, which implies our claim (by taking $p>3$).
\end{proof}

\subsection{Exotic $\R^{2n-1}$}
For the remainder, fix some $n>2$, and let $(M,\xi)$ and $\alpha$ be a contact $(2n-1)$-manifold and contact form as in Proposition~\ref{propn:ustilovsky-sphere-degrees-1mod4} (when $2n-1=4m+1$) or Proposition~\ref{propn:ustilovsky-sphere-degrees-3mod4} (when $2n-1=4m+3$). In either case, $M$ is homeomorphic to $\s^{2n-1}$, and hence, if we remove a point from $M$, the result is (by \cite{stallings-no-exotic-rn}) diffeomorphic to $\R^{2n-1}$.

\begin{lem}
  The resulting contact structure on $M\setminus\set{\pt}\cong\R^{2n-1}$ is tight.
\end{lem}
\begin{proof}
  This follows from $M$ itself being tight, since an open subset of a tight contact manifold is again tight. That the Brieskorn manifold $M$ is tight is well-known and follows from its being \emph{symplectically fillable} (see \cite[p.~284]{borman-eliashberg-murphy-overtwisted}), which in turn follows from its being holomorphically fillable: each Brieskorn manifold is the link of an isolated singular point of the zero set a certain holomorphic function, and by slightly perturbing this function, one obtains a smooth variety, and hence a holomorphic filling (\textit{cf}. \cite[\S4]{brieskorn-beispiele-von-sing}) of a contactomorphic manifold (by Gray's stability theorem).
\end{proof}

\begin{thm}\label{thm:exotic}
  The tight contact manifold $M\setminus\set{\pt}\cong\R^{2n-1}$ is not contactomorphic to $\R^{2n-1}_\std$. Moreover, there are infinitely many pairwise non-contactomorphic tight contact structures on $\R^{2n-1}$.
\end{thm}
\begin{proof}
  We will show that $M\setminus\set{\pt}$ has non-trivial contact homology in some positive degree, which by Proposition~\ref{propn:Rn-triv-contact-homology} implies the first claim.

  As a different model for $M\setminus\set{\pt}$, we take the \emph{contact connected sum} $M\#\R^{2n-1}_\std$ (see \cite[Chapter~5]{ustilovsky-thesis}). This is contactomorphic to $M\setminus\set{\pt}$ since the standard contact sphere $\s^{2n-1}_\std$ is a unit for the contact connected sum, and since $\R^{2n-1}_\std\cong\s^{2n-1}_\std\setminus\set{\pt}$.

  Next, let us equip $\R^{2n-1}_\std$ with a contact form as in Lemma~\ref{lem:Rn-no-reeb-orbits} and $M$ with a contact form as in Proposition~\ref{propn:ustilovsky-sphere-degrees-1mod4}~or~\ref{propn:ustilovsky-sphere-degrees-3mod4}. We then have by \cite[Theorem~5.2.1]{ustilovsky-thesis} that, for any $N$, there is a contact form $\alpha$ on $M\#\R^{2n-1}_\std$ such that, for each even $k\le N$, the number of Reeb orbits of homological degree $k$ is as given in Propositions~\ref{propn:ustilovsky-sphere-degrees-1mod4}~and~\ref{propn:ustilovsky-sphere-degrees-3mod4},  there is exactly one Reeb orbit of homological degree $k$ for odd $2n-3\le k<N$, and there are no Reeb orbits of homological degree $<2n-4$ (note that \emph{loc. cit.} is stated in terms of Conley-Zehnder index, not homological degree).

  The ``index-positivity'' required in \emph{loc. cit.} is established on \cite[p.~67]{ustilovsky-thesis} in the $4m+1$ case, and the same argument given there also applies in the $4m+3$ case. The form on $\R^{2n-1}_\std$ is trivially index-positive since it has no Reeb orbits.

  We also note that the contact form $\alpha$ constructed in the just-cited theorem agrees with the given contact form on $\R^{2n-1}_\std$ outside of a compact set, and is in particular still admissible with respect to a convex structure. Hence, we may compute $\CH_\bullet(M\#\R^{2n-1}_\std,\xi)$ (with $\xi=\ker\alpha$) as the homology of $\CC_\bullet\defeq\CC_\bullet(M\#\R^{2n-1}_\std,\xi)_{\alpha}$ with respect to some differential, and we would like to show that the resulting homology is non-trivial (in fact, for \emph{any} differential).

  We have that (for sufficiently large choice of $N$ above) $\CC_\bullet$ has the number $c_k$ of generators of degree $k$ for $k=1,\ldots,K$ displayed below (with $K=4n-7$ in the case $2n-1=4m+1$ and $K=4n-9$ in the case $2n-1=4m+3$).
\[
  \begin{cases}
    c_k=0&\text{for }0<k<2n-4\\
    c_k=1&\text{for }2n-4\le k\le 4n-9\\
    c_{4n-8}=3\\
    c_{4n-7}=2
  \end{cases}
  \quad\quad
  \begin{cases}
    c_k=0&\text{for }0<k<2n-4\\
    c_k=1&\text{for }2n-4\le k\le 4n-11\\
    c_{4n-10}=2\\
    c_{4n-9}=1
  \end{cases}
\]
This is easily seen by recalling that the generators of degree $k$ are in bijection with multisets of ``good'' (see \cite[\S1.2]{pardon-contact-homology-and-vfc}) Reeb orbits with homological degrees summing to $k$, and with each odd-degree orbit occurring with multiplicity at most one, and referring to the number of Reeb orbits of $\alpha$ in each homological degree described above. Here, all of the Reeb orbits are good since, by the computation in \cite[Lemma~4.2]{ustilovsky-infinitely-many-spheres}, none of the eigenvalues of their linearized return maps are real.

  We now claim that any chain complex $(C_\bullet,\partial)$ with the number of generators in the given degrees displayed above has non-trivial homology in a positive a degree, which we prove by contradiction. We treat the $2n-1=4m+1$ case; the argument in the $4m+3$ case is similar.

  Suppose that $\Hm_k(C_\bullet)=0$ for all $k>0$. Now, the generator in degree $2n-4$ must be a cycle. Hence, since $\Hm_{2n-4}(C_\bullet)=0$, the differential $\partial_{2n-3}\colon C_{2n-3}\to C_{2n-4}$ must be an isomorphism. Since $\partial^2=0$, it follows that $\partial_{2n-2}=0$. Continuing in this way, we find that $\partial_k=0$ for each even $2n-2\le k\le 4n-8$. Hence all three generators in $C_{4n-8}$ are cycles, but they cannot all be boundaries since $C_{4n-7}$ is 2-dimensional.

  To prove the last statement of the theorem, we consider the contact connected sum $(\#_{i=1}^rM)\#\R^{2n-1}_\std$ of $r$ copies of $M$ and $\R^{2n-1}_\std$, which is again diffeomorphic to $\R^{2n-1}$. This is still tight since the contact connected sum of symplectically fillable contact manifolds is still symplectically fillable. Again appealing to \cite[Theorem~5.2.1]{ustilovsky-thesis}, we can find a contact form $\alpha$ with $2r$ Reeb orbits in homological degree $K-1$ (with $K$ either $4n-7$ or $4n-9$) and $r$ Reeb orbits in each other homological degree $2n-4\le k\le K$.
  Thus, the number $c_k$ of generators in each degree is given by
  \[
    \begin{cases}
      c_k=0&0<k<2n-4\\
      c_k=r&2n-4\le k\le 4n-9\\
      c_{4n-8}=2r+\frac{r(r+1)}2\\
      c_{4n-7}=r+r^2
    \end{cases}
    \quad
    \begin{cases}
      c_k=0&0<k<2n-4\\
      c_k=r&2n-4\le k\le 4n-11\\
      c_{4n-10}=2r\\
      c_{4n-9}=r.
    \end{cases}
  \]
  Let us consider the case $2n-1=4m+1$; again, the other case is similar.

  We know a little more than just the number of generators in each degree: we have decompositions $\CC_{4n-8}=\CC^{\new}_{4n-8}\oplus\CC^{\old}_{4n-8}$ and $\CC_{4n-7}=\CC^{\new}_{4n-7}\oplus\CC^{\old}_{4n-7}$, with $\dim(\CC^{\new}_{4n-8})=2r$ and $\dim(\CC^{\new}_{4n-7})=r$, and moreover $\partial(\CC^{\old}_{4n-7})\subset\CC^{\old}_{4n-8}$.

  We claim that for such a chain complex $(C_\bullet,\partial)$, we have $\sum_{k=2n-4}^{4n-8}\dim\Hm_k(C_\bullet)\ge r$, which implies our claim.

  We can prove this, for example, as follows. Consider the truncation $C'_\bullet$ of $C_\bullet$ to degrees $k\le 4n-8$, and in which we only keep the summand $C^{\new}_{4n-8}$ of $C_{4n-8}$.
  The Euler characteristic $\chi(C'_\bullet)$ is $2r$, and hence $\sum_{k=2n-4}^{4n-8}\dim\Hm_k(C'_\bullet)\ge 2r$.
  But since $\dim(C^\new_{4n-7})=r$ and $\partial(C^{\old}_{4n-7})\subset C^\old_{4n-8}$, it then follows that $\sum_{k=2n-4}^{4n-8}\dim\Hm_k(C_\bullet)\ge r$ as desired.
\end{proof}

\subsection{Hypertight contact manifolds}
The same approach as above can be used to prove the existence of infinitely many tight contact structures on $V\setminus\set{\pt}$ whenever $V$ is a \emph{hypertight} contact manifold with vanishing (or more generally torsion) first Chern class.

We recall:
\begin{defn}
  A contact manifold $(V,\xi)$ is \defword{hypertight} if it admits a Reeb vector field with no contractible Reeb orbits.
\end{defn}
By \cite{albers-hofer-weinstein-higher-dim}, every hypertight contact manifold is tight.
\begin{defn}
  The first Chern class $\ch_1(\xi)$ of a contact structure is the first Chern class of the complex vector bundle $(V,J)$, where $J$ is a $\dv\alpha$-compatible complex structure for some contact form $\alpha$ representing $\xi$ (this is independent of the choice of $\alpha$ and $J$).
\end{defn}

We give an example to show that such contact manifolds exist.
Recall that a symplectic form $\omega$ on $M$ is is \emph{symplectically aspherical} if $\int_{\s^2}u^*\omega=0$ for any $u\colon\s^2\to M$, and let us call $\omega$ \emph{semi-monotone} if the first Chern class $\ch_1(TM)$ (with respect to a compatible almost-complex structure) is a multiple $\lambda[\omega]$ of the de Rham cohomology class $[\omega]$.\footnote{
  In particular, we may have $\ch_1(TM)=0$.
  We recall that the case $\lambda>0$ is called \emph{monotone}.
  We do not know of an established term for what we call semi-monotone.
  The case $\lambda<0$ has been called \emph{negative monotone}.
}
The class of semi-monotone and symplectically aspherical closed symplectic manifolds includes $\T^{2n}$ -- and more generally contains surfaces of positive genus and is closed under products -- as well as varieties with ample canonical bundle and vanishing $\pi_2$, such as fake projective planes.
\begin{propn}
  Suppose $(M,\omega)$ is a closed, semi-monotone, symplectically aspherical symplectic manifold, with $\omega$ integral, i.e., $[\omega]\in\im(\Hm^2(M;\Z)\to\Hm^2_{\mathrm{dR}}(M);\R)$.
  Then the associated prequantization space $(V,\xi)$ is hypertight and satisfies $\ch_1(\xi)\cdot\Hm_2(X;\Z)=\set{0}$ (i.e., $\ch_1(\xi)$ is torsion).
\end{propn}

\begin{proof}
  Recall that the prequantization space (or \emph{Boothby-Wang construction}) associated to $(M,\omega)$ is the principle $\s^1$-bundle $\pi\colon V\to M$ with Euler class $[\omega]\in\Hm^2(M;\Z)$.
  The contact structure on $V$ arises as the connection one-form of an arbitrary principal connection.
  For details, see \cite[\S7.2]{geiges-intro-contact}.

  The Reeb vector field $\Rb_\alpha$ is simply given by the circle action on the fibers, hence all Reeb orbits are in the homotopy class of the fiber.
  In order for this to be non-trivial, it suffices by the fibration sequence
  \[
    \cdots\to\pi_2(V)\to\pi_2(M)\to\pi_1(\s^1)\to\pi_1(V)\to\cdots
  \]
  that $\pi_*\colon\pi_2(V)\to\pi_2(M)$ be surjective, and this corresponds to the restriction $u^*V$ of $\pi\colon V\to M$ to every sphere $u\colon\s^2\to V$ being trivial, which is equivalent (by the naturality of the Euler class) to the symplectic asphericity condition.

  Finally, we have that $\xi\cong\pi^*\Tn M$ and hence $\ch_1(\xi)=\pi^*\ch_1(\Tn M)$.
  By the Gysin sequence
  \[
    \cdots\to\Hm^0(M;\Z)\tox{[\omega]\cup}\Hm^2(M;\Z)\tox{\pi^*}\Hm^2(V;\Z)\to\cdots
  \]
  we thus have that $\ch_1(\xi)$ is torsion if and only if $(M,\omega)$ is semi-monotone.
\end{proof}

\begin{thm}\label{thm:hypertight}
  If $V$ is a manifold of dimension $>3$ admitting a hypertight contact structure $\xi$ with $\ch_1(\xi)\cdot\Hm_2(V;\Z)=\set{0}$, then $V$ and $V\setminus\set{\pt}$ both admit infinitely many pairwise non-contactomorphic tight contact structures.
\end{thm}

\begin{proof}
  The main new ingredient needed here is to use the variant $\CH^\contr$ of contact homology generated only by the \emph{contractible} Reeb orbits; see \cite[\S1.8]{pardon-contact-homology-and-vfc}.
  This variant immediately carries over to the present convex open case.

  We now obtain contact forms on $V\setminus\set{\pt}$ by starting with a given hypertight contact form on $V$ and then forming the contact connected sum $V\#W_r$ with the contact manifolds $W_r\defeq(\#_{i=1}^rM)\#\R^{2n-1}_\std$ from the previous section.
  Since $\Hm_2(W_r;\Z)\cong0$, it follows (from Mayer-Vietoris and naturality of the Chern clsas) that the resulting contact structure $\xi$ on $V\#W_r$ again satisfies $\ch_1(\xi)\cdot\Hm_2(V;\Z)=\set{0}$, and hence $\CH^\contr(V\#W_r)$ receives a $\Z$-grading (again, see \cite[\S1.8]{pardon-contact-homology-and-vfc}).

  Since $V$ has by assumption no contractible Reeb orbits, and all of the additional orbits appearing in the contact connected sum are contractible, the computation of the number of generators of $\CH^\contr(V\#W_r)$ proceeds precisely as in the previous section, and so in the same way we see that there are infinitely many non-isomorphic contact manifolds among the $V\#W_r$.

  The claim about $V$ (as opposed to $V\setminus\set{\pt})$ is proven in exactly the same way, by considering the contact manifolds $V\#(\#_{i=1}^rM)$ (and using contact homology for closed manifolds).
\end{proof}

\printbibliography
\end{document}